\documentclass[12pt]{amsart}
\usepackage{amscd,times,fullpage}

\usepackage[utf8]{inputenc}
\usepackage[T1]{fontenc}

\usepackage{amsfonts,amssymb,enumerate,amsthm,amsmath,color,graphicx}
\usepackage{hyperref}
\usepackage[all]{xy}

\newtheorem{theor}{Theorem}[section]
\newtheorem{prop}[theor]{Proposition}
\newtheorem{defin}{Definition}
\newtheorem{lem}[theor]{Lemma}
\newtheorem{cor}[theor]{Corollary}

\newtheorem{example}{Example}
\newtheorem{remark}{Remark}

\newcommand{\R}{\mathbb{R}}
\newcommand{\N}{\mathbb{N}}
\newcommand{\Hyp}{\mathbb{H}}
\newcommand{\C}{\mathbb{C}}
\newcommand{\T}{\mathbb{T}}
\newcommand{\bC}{\mathbb{C}}

\newcommand{\Z}{\mathbb{Z}}

\newcommand{\scal}{{\rm scal}}
\newcommand{\de}{\partial}
\newcommand{\deb}{\bar\partial}
\newcommand{\W}{\Omega}

\newcommand{\w}{\omega}

\newcommand{\ov}[1]{\overline{#1}}

\newcommand{\K}{K\"ahler}

\newcommand{\Bl}{\operatorname {Bl}}

\newcommand{\wk}{\w}

\newcommand{\hs}{\hspace{0.1em}}

\begin{document}

\title[Kähler duality and projective embeddings]{Kähler duality  and projective embeddings}

\author{Andrea Loi}
\address{(Andrea Loi) Dipartimento di Matematica \\
         Universit\`a di Cagliari (Italy)}
         \email{loi@unica.it}

\author{Roberto Mossa}
\address{(Roberto Mossa) Dipartimento di Matematica \\
         Universit\`a di Cagliari (Italy)}
        \email{roberto.mossa@unica.it}
        
        \author{Fabio Zuddas}
\address{(Fabio Zuddas) Dipartimento di Matematica \\
         Universit\`a di Cagliari (Italy)}
        \email{fabio.zuddas@unica.it}

\thanks{
The authors are supported by INdAM and  GNSAGA - Gruppo Nazionale per le Strutture Algebriche, Geometriche e le loro Applicazioni, by GOACT and MAPS- Funded by Fondazione di Sardegna, by the Italian Ministry of Education, University and Research through the PRIN 2022 project "Developing Kleene Logics and their Applications" (DeKLA), project code: 2022SM4XC8  and partially funded by PNRR e.INS Ecosystem of Innovation for Next Generation Sardinia (CUP F53C22000430001, codice MUR ECS00000038).}


\begin{abstract}
Motivated by the duality  theory between  Hermitian symmetric spaces
of noncompact and compact types, we introduce and examine the concept of \K\ duality between exact domains of $\C^n$
(cf. Definitions \ref{exactdomain} and \ref{dualcomplexdom}).

Specifically, we address the following question  (cf.  Definition \ref{FScomp}).

\vskip 0.1cm
\noindent
{\bf Question.} {\em Let $(U, g)$ be  an exact domain 
admitting  a \K\ dual $(U^*, g^*)$. Assume that there exists $\alpha>0$
such that  the following conditions hold:
\begin{itemize}
\item [(a)] 
$(U, \alpha g)$ has a Fubini-Study completion;
\item [(b)]
$(U^*, \alpha g^*)$ has a Fubini-Study compactification.
\end{itemize}
 Is it true that  $(U, g)$  is 
 biholomorphically isometric to  a bounded symmetric  domain?
}

\vskip 0.1cm

In Theorem \ref{mainteor},
we affirmatively answer this question for 
a Cartan-Hartogs domain $M_{\Omega, \mu}$ equipped with an appropriate    metric $g_{\Omega, \mu}$.
In Theorem \ref{mainteor2}, we provide a counterexample by introducing another  \K\ metric $\hat g_{\Omega, \mu}$ on 
$M_{\Omega, \mu}$, that  satisfies both conditions  (a) and (b). 
Furthermore, in Theorem \ref{mainteor3}, we characterize Hermitian symmetric spaces of compact type among classical flag manifolds based solely
on \K\ duality. We also propose and  thoroughly investigate a  conjecture  that our question has a positive answer if  the metric $g$ is \K-Einstein.
 \end{abstract}
 
\maketitle

\tableofcontents

\section{Introduction}
A Hermitian symmetric space of noncompact type   is  a 
\K\ manifold 
$(M, g)$ characterized by the property that,  for all $p\in M$, the geodesic symmetry
$$s_p: \exp_p (v)\mapsto\exp_p (-v), \forall v\in T_pM$$
 is a  globally defined  holomorphic isometry of $(M, g)$.
 
Up to homotheties, $(M, g)$  is  biholomorphically isometric to  a bounded  symmetric domain $\Omega\subset \C^n$ centred at  the origin $0\in \C^n$ equipped with  the  \K\ metric $g_\Omega$ whose associated  \K\ form is 
\begin{equation}\label{defomom}
\omega_\Omega=-\frac{i}{2\pi}\partial\bar\partial\log N_\Omega
\end{equation}
where 
\begin{equation}\label{genericnorm}
N_\Omega(z, \overline{z})=\left(V(\Omega)K_\Omega(z, \overline{z})\right)^{-\frac{1}{\gamma}}
\end{equation}
is the {\em generic norm},
$V(\Omega)$ is the {\em Euclidean volume} of $\Omega$, $\gamma$ the {\em genus} of $\Omega$ and $K_\Omega$ its {\em Bergman kernel}.

Thus 
\begin{equation}\label{linkmetrics}
g_\Omega = \frac{1}{\gamma} g_B
\end{equation}
where $g_B$ is the {\em Bergman metric} on $\Omega$.

Notice that the genus 
$\gamma$ of $\Omega$ is nothing but the Einstein constant 
of the \K-Einstein (KE in the sequel) metric $g_\Omega$.

Every Hermitian symmetric space of noncompact type is  a homogenous 
\K\ manifold
and there exists a complete classification of the irreducible ones, known as  {\em Cartan domains}, which include four classical series and two exceptional cases of complex dimensions $16$ and $27$, respectively.

A Cartan domain $\Omega$ is uniquely determined by a triple of integers $(r,a,b)$,
where $r$ 
represents the rank of $\Omega$, namely the maximal dimension of a complex totally geodesic submanifold of $\Omega$,
and $a$ and $b$ are positive integers such that 
\begin{equation}\label{genus}
\gamma=(r-1)a+b+2
\end{equation}
and
$$n = r + \frac{r (r-1)}{2} a + r b,$$
where  $n$ is the dimension of  $\Omega$ .
The {\em Wallach set} $W(\Omega)\subset {\R}$ of a Cartan  domain $\Omega\subset \C^n$ is a subset of ${\R}$ 
which  depends on  $a$ and   $r$.
More precisely we have
\begin{equation}\label{wallach}
W(\Omega)=\left\{0,\,\frac{a}{2},\,2\frac{a}{2},\,\dots,\,(r-1)\frac{a}{2}\right\}\cup \left((r-1)\frac{a}{2},\,+\infty\right).
\end{equation}

The importance of the Wallach set for our aim is due to the following  fact  (see the beginning of Section \ref{secproof1} for the definition of infinitely projectively induced \K\ metric).

\vskip 0.3cm

\noindent
{\bf Theorem A.\cite[Th.2]{LOIZEDDA2011KEsubProjInf}}
{\em Let $(\Omega, g_\Omega)$
be a Cartan domain. The \K\  metric $\alpha g_\Omega$ is infinitely projectively induced iff $\alpha\in W(\Omega)\setminus\{0\}$.
Consequently, if  $(\Omega, g_\Omega)$
is a bounded symmetric domain, $\alpha g_\Omega$  is
infinitely projectively induced for all sufficiently large $\alpha$.}

\vskip 0.3cm

One of the remarkable aspects  of the theory of Hermitian symmetric spaces  is the concept of {\em duality},
which provides a way to transition from a Hermitian symmetric space of noncompact type  $(\Omega, g_\Omega)$ to a Hermitian symmetric space of compact type
$(\Omega_c, g_{\Omega_c})$, known as the {compact dual} of $(\Omega, g_\Omega)$ (see, e.g. \cite{DISCALALOI2008sympdual} and references therein).

An interesting  feature of this   duality  is reflected in the following chain of  holomorphic embeddings, which hold  for  any bounded symmetric domain $\Omega$:
\begin{equation}\label{embeddings}
0\in\Omega\subset\C^n\stackrel{J}{\rightarrow} \Omega_c\stackrel{BW}{\rightarrow} \C P^d,
\end{equation}
where $J$ is  a holomorphic embedding with dense image   and  $BW$  is  the {\em Borel-Weil} embedding
satisfying $BW^*(g_{FS})=g_{\Omega_c}$ (see  \cite{NAKAGAWATAKAGI1976locsymm} for more details on the value of $d$).
Moreover, 
$\Omega_c=J(\C^n)\sqcup H$, 
where $H=BW^{-1}(\{Z_0=0\})$, and  where 
$[Z_0, \dots ,Z_d]$ denotes the homogeneous coordinates on $\C P^d$. 
The  embeddings \eqref{embeddings}  have been extensively used  to study the symplectic geometry of bounded symmetric domains and theirs duals in \cite{DISCALALOI2008sympdual}, \cite{DISCALALOIROOS2008bisympl}, \cite{LOIMOSSA2011DistExp} (see also \cite{CHANMOK2017holisom} for  interesting rigidity properties of holomorphic isometries of the complex hyperbolic space into bounded symmetric domains). 

If $\omega_{\Omega_c}$ is the \K\ form associated to $g_{\Omega_c}$ it turns out that 
$${J^*\omega_{\Omega_c}}=\frac{i}{2\pi}\partial\bar\partial\log N_\Omega^*,$$
where we define
\begin{equation}\label{Nstar}
N_\Omega^*(z, \bar z):=N_\Omega (z, -\bar z)
\end{equation}
and where  $N_\Omega (z, \bar z)$ is the generic norm of $\Omega$ given by \eqref{genericnorm}.

\begin{example}\rm\label{projspace}
The prototypical  example of Hermitian symmetric space of noncompact type is  the unit ball
$\C H^n=\{z\in \C^n \ | \ |z|^2<1\}$
equipped with the hyperbolic metric $g_{\C H^n}=g_{hyp}$
such that 
$$\omega_{hyp}=-\frac{i}{2\pi}\partial\bar\partial\log N_{\C H^{n}}=-\frac{i}{2\pi}\partial\bar\partial\log (1-|z|^2)$$
In this case $a=2, b=n-1$, $r=1$ and $\gamma =n+1$.
The compact dual of 
$(\C H^n, g_{hyp})$ is $(\C P^n, g_{FS})$ the complex projective space with the Fubini-Study metric  $g_{FS}$.
In the affine chart $U_0=\left\{Z_0\neq 0\right\}\subset\C P^n$,  the Fubini-Study form,  with respect the affine coordinates $z_j=\frac{Z_j}{Z_0}$, is given by
$\omega_{FS}=\frac{i}{2\pi}\de\deb \log \left(1+ |z|^2\right)$. 
Notice that in this  case  the $BW$ embedding is the identity map of $\C P^n$.  
\end{example}

Using the fact that $g_{\Omega_c}$ is projectively induced  by the Borel-Weil embedding  and  $\omega_{\Omega_c}$ is such that $\omega_{\Omega_c}(B)=\pm\pi$,  where $B$ is a generator of 
$H_2(M, \Z)$ (see e.g. \cite{LOIMOSSAZUDDAS2015SymplCap}) we derive  the following compact version of Theorem A.

\vskip 0.3cm

\noindent
{\bf Theorem B.}
{\em Let $(\Omega_c, \alpha g_{\Omega_c})$ be a Hermitian symmetric space of compact type, with $\alpha> 0$. Then  $\alpha g_{\Omega_c}$ is  projectively induced 
iff $\alpha$ is a  positive integer.}

\vskip 0.3cm

Theorems A and B prompt the question of whether the concept of duality is specific to Hermitian symmetric spaces and, in particular, whether it characterizes them. To rigorously address this question (see Question A below), it is natural to begin by introducing some definitions.

\begin{defin}\label{FScomp}
A  compact (resp. complete) \K\  manifold $(\tilde M, \tilde g)$
is a {\em Fubini-Study  compactification (resp. completion)}
of a  \K\ manifold $(M, g)$  if there exists 
a  holomorphic isometric embedding  $J:(M, g)\rightarrow (\tilde M, \tilde g)$ such that $J(M)$
is an open and dense subset  of  $\tilde M$
and $\tilde g$ is finitely (resp. infinitely) projectively induced.  
\end{defin}

\begin{remark}\label{rmkcompl}\rm
Although a compact manifold  is (of course)  complete, Calabi's rigidity theorem \cite[Th.9]{CALABI1953diast} implies that  a \K\ manifold admitting  a Fubini-Study compactification does {\em not} admit a Fubini-Study completion, and viceversa. 
\end{remark}

\begin{example}\rm 
Being $g_{\Omega}$ complete it follows from  Theorem A that 
a bounded symmetric domain $(\Omega, \alpha g_\Omega)$ admits a Fubini-Study completion
iff $\alpha\in W(\Omega)\setminus \{0\}$, 
by trivially setting $(\tilde\Omega, \alpha \tilde g_\Omega) =(\Omega, \alpha g_\Omega)$ and $J=id_\Omega$.
On the other hand,
by Theorem B,
$(\C^n, J^*\alpha g_{\Omega_c})$ with $\alpha\in\Z^+$  admits a Fubini-Study compactification
where  $J:\C^n\rightarrow \Omega_c$ is  the embedding in \eqref{embeddings}.
\end{example}

It is not hard to see that  the \K\ potentials
$-\log N_\Omega$ and 
$\log N_\Omega^*$ (see \eqref{Nstar})
are the {\em Calabi's diastasis functions} 
$D_0^{g_\Omega}$ and $D_0^{J^*g_{\Omega_c}}$
  at the origin $0\in \Omega\subset \C^n$ for  the metrics $g_\Omega$ and  $J^*g_{\Omega_c}$, respectively (see  \cite{CALABI1953diast} or Section \ref{FLAG} below).

Taking inspiration from this fact and by the duality theory  between hermitian symmetric spaces, we propose the following definitions (cf. also
\cite{LOIDISCALA2010sympdualitycomplex}).

\begin{defin}\label{exactdomain}\rm
A pair  $(U, g)$ is called an {\em exact domain}
if $U$  is a 
connected complex  domain of $\C^n$ containing the origin 
and  $g$ is a real analytic \K\ metric whose associated  
\K\ form is given by 
$\omega=\frac{i}{2\pi}\partial\bar\partial D^g_0$,
where $D^g_0$ is Calabi's diastasis function for $g$ at $0\in\C^n$ 
and  $D^g_0$ cannot be extended to an open subset $V$ strictly containing $U$.
\end{defin}

\begin{remark}\rm\label{rmkcompld}
The exploration of the maximal extension domain of Calabi's diastasis function is a complex and profound issue (see \cite{CALABI1953diast} and also  the recent preprint 
\cite{loi2023ricciiterationswellbehavedkahler}). 
 All examples at our disposal of  exact domains $(U, g)$ with $g$ infinitely projectively induced  are  complete. Thus, such domains  admit  a Fubini-Study completion by simply setting $(\tilde U, \tilde g)=(U, g)$ and $J=id_U$.
\end{remark}

\begin{defin}\label{dualcomplexdom}\rm
Let $(U, g)$ and $(U^*, g^*)$ be two exact domains
of $\C^n$.
We say that $(U^*, g^*)$ is the   {\em \K\ dual} of $(U, g)$ and $g^*$ is a \K\ metric dual to $g$  if 
\begin{equation}\label{diastasisstar}
D_0^{g^*}(z, \bar z)=-D_0^{g}(z, -\bar z), \forall z\in U\cap U^*.
\end{equation}
\end{defin}

One can easily construct explicit examples of exact domains  not admitting 
 a \K\ dual (cf. Section  \ref{FLAG}).

\begin{example}\rm\label{boundedbochner}
Any bounded symmetric domain $(\Omega, g_\Omega)$ is an exact domain
and its  \K\ dual is the exact domain given by 
$(\Omega^*, g^*_\Omega)=(\C^n, J^*g_{\Omega_c})$, where  $g_{\Omega_c}$ is the \K\ metric on the compact dual $\Omega_c$
of $\Omega$  and $J:\C^n\rightarrow \Omega_c$ is the embedding  in \eqref{embeddings}.
\end{example}

Inspired by Theorems A and Theorem B,  Definitions \ref{FScomp},  \ref{exactdomain},  \ref{dualcomplexdom}  and Example \ref{boundedbochner} we address the following

\vskip 0.3cm

\noindent
{\bf Question A.} {\em Let  $(U, g)$ be an exact domain 
admitting a \K\ dual $(U^*, g^*)$. Suppose   that there exists $\alpha>0$
for which the following properties hold true:
\begin{itemize}
\item [(a)] 
$(U, \alpha g)$ has a Fubini-Study completion;
\item [(b)]
$(U^*, \alpha g^*)$ has a Fubini-Study compactification.
\end{itemize}
 Is it true that  $(U, g)$  is  
 biholomorphically isometric to  a bounded symmetric  domain $(\Omega,  \lambda g_\Omega)$,
 for some $\lambda\in\R^+$
(and hence  $\lambda\alpha\in W(\Omega)\cap \Z^+$)?}

\vskip 0.3cm

In this paper we consider Question A in the context where $U$ is a {\em Cartan-Hartogs domain} 
(CH domain   in the sequel) equipped with two  natural \K\ metrics $g_{\Omega, \mu}$
and $\hat g_{\Omega, \mu}$ (see below).
CH domains   are a 1-parameter family of noncompact  domains of $\mathbb{C}^{n+1}$, given by:
\begin{equation*}
M_{\Omega, \mu}:=\left\{(z, w) \in \Omega \times\mathbb{C}\, |\, \left|w\right|^2<N_{\Omega}^\mu(z, \bar{z})\right\}
\end{equation*}
where $\Omega \subset \mathbb{C}^n$ is a Cartan domain,  
known as  the base of $M_{\Omega, \mu}$,   $N_{\Omega}(z, \bar{z})$ is its generic norm,  and  $\mu>0$ is a positive real parameter.

 The first metric on $M_{\Omega, \mu}$ that we consider 
 is the  \K\   metric $g_{\Omega, \mu}$,
 introduced and studied by  G. Roos, A. Wang, W. Yin, L. Zhang, W. Zhang in \cite{WANGYINZHANGROSS2006KEonHartogs} and \cite{WANGYINZHANGZHANG2004KEonNNhom}.
The  associated \K\ form is given by
\begin{equation}\label{eqKobform}
\wk_{\Omega, \mu}=-\frac{i}{2\pi}\de\deb \log \left(N_{\Omega}^\mu(z, \bar{z})-|w|^2\right).
\end{equation}


One of  the important features of this metric  is the following result
which shows that for suitable values of the parameter $\mu$
it is KE.

\vskip 0.3cm
\noindent
{\bf Theorem C.}(\cite{WANGYINZHANGROSS2006KEonHartogs} and \cite{WANGYINZHANGZHANG2004KEonNNhom})
{\em   Let $M_{\Omega, \mu}$ be a  CH domain. The metric 
$g_{\Omega, \mu}$ is KE
iff  $\mu =\frac{\gamma}{n+1}$, where $n$ is the complex dimension of $\Omega$ and $\gamma$ its genus.}
\vskip 0.3cm


The following theorem represents the first result of this paper. It provides a positive answer to Question A for a CH domain
$(M_{\Omega, \mu},  g_{\Omega, \mu})$ and  also characterizes the complex hyperbolic space among CH domains.

\begin{theor}\label{mainteor}
The  CH domain 
$(M_{\Omega, \mu}, g_{\Omega, \mu})$
is an exact domain that admits a \K\  dual $(M_{\Omega, \mu}^*=\C^{n+1}, g_{\Omega, \mu}^*)$.
Moreover, $(M_{\Omega, \mu}, \alpha g_{\Omega, \mu})$
admits a Fubini-Study completion for all  sufficiently large $\alpha$. 
Additionally,  $(\C^{n+1}, \alpha g_{\Omega, \mu}^*)$ admits a Fubini-Study compactification for some $\alpha$
iff $M_{\Omega, \mu}$ is homogeneous, i.e.  $(M_{\Omega, \mu}, \alpha g_{\Omega, \mu})=(\C H^{n+1}, g_{hyp})$.
 \end{theor}
 \begin{remark}\rm
 It is worth pointing out that the \K\ dual $(\C^{n+1}, \alpha g_{\Omega, \mu}^*)$
 of $(M_{\Omega, \mu}, g_{\Omega, \mu})$
 has also been examined in \cite{MOSSAZEDDAS2022symplch}
 from a symplectic perspective, with the aim  to extend the symplectic duality 
 given in \cite{DISCALALOI2008sympdual} for symmetric spaces  to CH domains.
 \end{remark}

The second metric $\hat g_{\Omega, \mu}$ on a  CH domain  $M_{\Omega, \mu}$ that we consider in this paper 
 is the \K\ metric whose associated \K\ form is given by (see Section \ref{secproof2} for details):

\begin{equation}\label{modmetric}
\hat \omega_{\Omega, \mu} =-\frac{i}{2\pi} \partial \bar\partial \log (N^\mu_\Omega(z, \bar z) - |w|^2) - \frac{i}{2\pi} \partial \bar\partial \log N_\W^\mu(z, \bar z) 
\end{equation}

In the following theorem, which represents  our second main result,  we show that the CH domain  $(M_{\Omega, \mu},\hat g_{\Omega, \mu})$
satisfies the assumptions (a) and (b) of Question A, for suitable values of  $\alpha$ and $\mu$. Therefore,  Question A has a negative answer  for these values.

\begin{theor}\label{mainteor2}
The CH domain 
$(M_{\Omega, \mu}, \hat g_{\Omega, \mu})$
is an exact domain that admits a \K\ dual $(M_{\Omega, \mu}^*=\C^{n+1}, \hat g_{\Omega, \mu}^*)$. 
For $\mu\in \Z^+$ and  sufficiently large integer $\alpha$, 
$(M_{\Omega, \mu},  \alpha\hat g_{\Omega, \mu})$
has a Fubini-Study completion 
and 
$(\C^{n+1}, \alpha\hat g_{\Omega, \mu}^*)$
has a Fubini-Study compactification. 
\end{theor}

This theorem draws us with  a strong analogy with the duality theory between  Hermitian  symmetric spaces.
Moreover, the existence of such a compactification   offers a bridge between the complex geometry of a given CH domain and the algebraic geometry of the  Fubini-Study
compactification of its \K\ dual.

\begin{remark}\rm
One natural metric to  consider on a  CH domain is, of course, the Bergman one (see \cite{ROOS2004bergman}). We can prove that although the Bergman metric on  CH domains admits a \K\ dual, 
the latter is not finitely projectively induced. This and other properties of the Bergman metric  on CH domains  and its \K\ dual will be explored  in a forthcoming paper.  
\end{remark}

A natural extension of the concept of Hermitian symmetric space of compact type is that of a  flag manifold $(F, g_F)$, i.e., compact and simply-connected homogeneous 
complex manifold $F$ equipped with a homogeneous \K\ metric $g_F$. Note that if the \K\ form $\omega_F$ associated to $g_F$ is integral 
then $g_F$ is finitely projectively induced (cf. \cite{TAKEUCHI1978hom} and \cite{DISCALLOIISHI2012Kimm}).

Moreover,
by restricting to the case of a  flag manifold of classical type, 
for  each $p\in F$ there exists a holomorphic embedding
with dense image $J_F:\C^n\rightarrow F$, $J_F(0)=p$.
This embedding is analogous to the embedding $J$ given by 
\eqref{embeddings} for a Hermitian symmetric space of compact type
(see  Section \ref{FLAG} below for the definition of $J_F$ and more details).

Thus a flag manifold of classical type $(F, g_F)$ of complex dimension $n$ and integral $\omega_F$ is a Fubini-Study compactification
of $(\C^n, J_F^*g_F)$. Consequently,  it is natural to investigate whether the latter admits  a \K\ dual.
In  the following theorem,
which represents our final main result,
we prove that this is the case only when the flag is indeed a Hermitian symmetric space of compact type.

\begin{theor}\label{mainteor3}
Let $(F, g_F)$ be a flag manifold of classical type.
If $(\C^n, J_F^*g_F)$ admits a \K\ dual,
then $(F, g_F)$
is biholomorphically isometric to 
a Hermitian symmetric space of compact type.
\end{theor}
It is worth pointing out that in Theorem \ref{mainteor3} we are not assuming any of the conditions (a) and  (b) of Question A. 
Therefore,  this theorem provide a characterization of Hermitian symmetric spaces of compact type among classical  flag manifolds solely  in terms
of \K\ duality.

\vskip 0.3cm
 
 The paper is organized as follows. In Sections \ref{secproof1} and \ref{secproof2},
we  describe  the values of $\alpha$ and $\mu$ for which  the metrics 
 $\alpha g^*_{\Omega, \mu}$, 
 $\alpha \hat g_{\Omega, \mu}$ and 
 $\alpha\hat g^*_{\Omega, \mu}$  
 are projectively induced (see Propositions \ref{propproj}, \ref{propmod} and \ref{propalphamu}, respectively). We also prove Theorems \ref{mainteor}
 and \ref{mainteor2}.
In the final  part  of Section \ref{secproof2}
we  show that the homogeneity assumption in  Theorem \ref{mainteor3} cannot be weakened to cohomogeneity  $1$ by showing that  the Fubini-Study compactification of  $(\C^{n+1}, \hat g^*_{\Omega, \mu})$ given in Theorem \ref{mainteor2}, 
is a cohomogeneity $1$ $G$-space.

In Section \ref{FLAG}, after recalling the construction of the embedding 
$J_F:\C^n\rightarrow F$ for a flag manifold of classical type and the definition of Alekseevsky-Peremolov coordinates,
we prove Theorem \ref{mainteor3}. 
In Section \ref{conjecture} we introduce a conjecture (Conjecture A) asserting that Question A has a positive answer if the metric $g$ is \K-Einstein.
We  also analyze the necessity of conditions (a), (b) and the KE assumption in Conjecture A.  
Specifically, we show the necessity of assumption (b) in Example \ref{HIDEYUKI}, examine  the role of assumption (a) by comparing it with  other long-standing  conjectures 
within the context of flag manifolds,  and show the necessity of the KE assumption  by proving that 
the metric $\hat g_{\Omega, \mu}$ is never KE for any value of $\mu$.
 In the last section (Section \ref{finrmk}) we show how 
our results on CH domains could be  extended to generalized CH domains.

\section{The \K\ dual of $(M_{\Omega, \mu}, g_{\Omega, \mu})$ and the proof of Theorem \ref{mainteor}}\label{secproof1}
Consider the \K\ dual  of a CH domain $(M_{\Omega, \mu}, g_{\Omega, \mu})$ 
namely the \K\ manifold $(M^*_{\Omega, \mu}, g^*_{\Omega, \mu})$, with associated \K\ form given by 
\begin{equation}\label{eqdualKob}
\omega^*_{\Omega, \mu}=\frac{i}{2\pi}\de\deb \log \left(N_{\Omega}^\mu(z, -\bar{z})+|w|^2\right).
\end{equation}
These domains were introduced in \cite{MOSSAZEDDAS2022symplch} where symplectic aspects were studied in analogy with those of 
\cite{DISCALALOI2008sympdual}.
In particular,  one can prove (see \cite[Section 3]{MOSSAZEDDAS2022symplch}) 
that $M^*_{\Omega, \mu}=\C^{n+1}$ and it can be easily  verified  that
$(M^*_{\Omega, \mu}, g^*_{\Omega, \mu})$
 is indeed the \K\ dual since $N_{\Omega}^\mu(z, -\bar{z})$
is real valued.

In the proof of  Theorem \ref{mainteor}
we need the following Propositions \ref{propproj} and \ref{ioemiki} and  Lemma \ref{lemmamain}. 

First we recall some basic facts. 
Throughout this paper we say that a \K\ metric $g$ on a complex manifold $M$ is {\em projectively induced} if  there exists a holomorphic map $f:M \to \C P^N$, $N \in \N \cup \left\{\infty\right\}$, such that $f^*g_{FS}=g$, where $g_{FS}$ is the Fubini-Study metric on $\C P^N$.
Such a map is called a {\em holomorphic isometry}.
Notice that we are not assuming either  $M$ compact or complete.
A \K\ metric  $g$ is said to be {\em finitely projectively  induced} if $N$ can be chosen finite, and  {\em infinitely projectively induced} if $N=\infty$ and the immersion is full, meaning its image is not contained within any finite dimensional complex projective space. For more on 
projectively induced \K\ metrics, 
refer to \cite{CALABI1953diast} and \cite{LoiZeddaBook2018diast}.

\begin{prop}\label{propproj}
The \K\ metric  $\alpha g^*_{\Omega, \mu}$ on $\C^{n+1}$
 is finitely projectively induced if and only if $\alpha, \mu  \in \Z^+$.  
\end{prop}
 \begin{proof}
 Let $(\Omega_c, g_{\Omega_c})$ be the compact dual of  the Cartan domain $(\Omega, g_\Omega)$.
We know that the  \K\ metric $g_{\Omega_c}$ is finitely  projectively induced via the Borel--Weil embedding
$BW: \Omega_c \rightarrow \mathbb{C} P^{d}, \ BW(p)=\left[s_0(p), \cdots  ,s_{d}(p)\right]$, i.e.
$BW^*g_{FS}=g_{\Omega_c}$
for a suitable $d=\dim H^{0}(L)-1$ depending on  $\Omega_c$ (cf. \eqref{embeddings}).
Here  $s_0, \dots ,s_d$ are global holomorphic sections of the 
holomorphic line bundle $L$ over $\Omega_c$ such that $c_1(L)=[\omega_{\Omega_c}]$. 
If  $\mu\in \Z^+$ then $\mu g_{\Omega_c}$ is finitely  projectively induced, i.e. there exists a holomorphic embedding
\begin{equation}\label{eqembdcarmeno}
F_{\mu}:\Omega_c\rightarrow \C P^{d_\mu}, \ F_{\mu}(p)=\left[s^{(\mu)}_0(p), \cdots  ,s^{(\mu)}_{d_\mu}(p)\right],
\end{equation}
such that $F_{\mu}^*g_{FS}=\mu g_{\Omega_c}$, 
where $s^{(\mu)}_0, \dots , s^{(\mu)}_{d_\mu}$
is a basis of  global holomorphic sections of $L^{\otimes\mu}$
and $d_\mu+1=\dim H^{0}(L^{\otimes\mu})$.
Moreover, 
we can also assume  $s^{(\mu)}_0(J(z)) \neq 0$, for all $z \in \mathbb{C}^n$ where $J:\C^n\rightarrow\Omega_c$ is given by \eqref{embeddings}.
The existence of $F_\mu$ with these properties 
follows by the fact that $L^{\otimes\mu}$ defines a regular quantization of the homogeneous \K\ manifold $(\Omega_c, g_{\Omega_c})$
(the reader is referred to  \cite[Th. 5.1]{AREZZOLOI2003quant}
for details).

 One then  easily deduces that the map
\begin{equation}\label{eqembdcarp}
\tilde{F}_\mu: \mathbb{C}^{n+1} \rightarrow \mathbb{C} P^{d_{\mu}+1}, \ (z, w) \mapsto\left[s^{(\mu)}_0(J(z)), \ldots, s^{(\mu)}_{d_\mu}(J(z)), ws^{(\mu)}_0(J(z)) \right],
\end{equation}
is a holomorphic isometric embedding of $\left(\C^{n+1}, g^*_{\W,\mu}\right)$ into $\left(\C P^{d_\mu+1}, g_{FS}\right)$.
If also  $\alpha$ is  a positive integer we get that 
$\alpha g^*_{\W,\mu}$ is finitely  projectively induced by the holomorphic map 
$V_\alpha\circ\tilde F_\mu$, where $V_\alpha : \C P^{d_\mu+1}\rightarrow \C P^{{\left({d_\mu+1+\alpha\atop d_{\mu}+1}\right)}}$
is Calabi's map (see \cite[Th.13]{CALABI1953diast}), namely a suitable normalization
of the Veronese map.

Conversely, assume that $\alpha g^*_{\Omega, \mu}$ is projectively induced.
Then the same will be true for the metric 
$\alpha g^*_{\C H^1,\mu}$ on  $\C^2$.
Indeed, by the hereditary property of the  Calabi's diastasis function  (see \cite[Prop.6]{CALABI1953diast}), 
we get that 
$$D_0^{g^*_{\Omega, \mu}}=\log \left(N^\mu_\Omega (z, -\bar z)+|w|^2\right)$$ on $\C^{n+1}$
restricts to $D_0^{g^*_{\C P^1, \mu}}=\log \left(\left(1+|\xi|^2\right)^\mu+|w|^2\right)$ on $\C^2$, where we choose a 
$(\C P^1, \mu g_{FS})$ complex and totally  geodesic embedded into $\Omega_c$.

Take now $\left\{\left(\xi,w\right)\in\C^2\mid \xi=0\right\}$ and  $\left\{\left(\xi,w\right)\in\C^2\mid w=0\right\}$, equipped with the 
\K\ metrics induced by $\alpha g^*_{\C H^1,\mu}$. It is immediate to see that they are  nothing but  $(\C , \alpha g_{FS})$ and  $(\C , \alpha \mu g_{FS})$ respectively, where $g_{FS}$ 
is the Fubini-Study metric of $\C P^1$ restricted to $\C=\{Z_0\neq 0\}\subset \C P^1$.
Thus,  the assumption that $\alpha g^*_{\Omega, \mu}$  (and hence $\alpha g^*_{\C H^1, \mu}$) is projectively induced implies
that both $\alpha g_{FS}$ and $\alpha \mu\hs g_{FS}$ are projectively induced.
This forces  $\alpha \in \Z^+$ and $\alpha \mu \in \Z^+$ and we can then write
$$\mu = \frac{a}{b} \quad \text{and} \quad \alpha=bk,$$
for some $a,b,k\in \Z^+$ and  $a,b$ coprime. Consider now  $\left\{\left(\xi,w\right)\in\C^2\mid w=1\right\}$ equipped with  the \K\ metric $g$ induced  by $\alpha g^*_{\C H^1, \mu}$.
The proof  will be completed if we show that $g$ is not projectively induced if $\frac{a}{b}\not \in \Z^+$. 
Notice that the \K\ form $\omega$ associated to $g$ is 
$$\w =\frac{i}{2\pi}\de\deb \Phi,\qquad \Phi (\xi)=bk\log\left(\left(1+|\xi|^2\right)^\frac{a}{b}+1\right).$$
It is immediate to  verify that $D(\xi)=\Phi(\xi)-bk\log 2$,  is the diastasis function for the metric $g$ centred at the origin.
Moreover, by using the radiality of   $D$ and  Calabi's criterium (see \cite{CALABI1953diast} or  \cite{LoiZeddaBook2018diast}), in order to show that $g$ is not projectively induced
 one needs to show that at least  one  the coefficients of the Taylor expansion  in $\xi$ and $\bar \xi$ around the origin
of $e^D-1$  is negative.  We can see that this is equivalent to show that the Taylor expansion of $\Psi(|\xi|^2)=\left(\left(1+|\xi|^2\right)^\frac{a}{b}+1\right)^{bk}$ in $|\xi|^2$, around the origin, has negative coefficients when $\frac{a}{b}\not \in \Z^+$.

This expansion reads as 
\begin{equation}\label{eqprlemp}\begin{split} 
\Psi(\left|\xi\right|^2)&= \sum_{p=0}^{kb}\binom{kb}{p}\left(1+\left|\xi\right|^2\right)^\frac{pa}{b}\\
 & =
 \sum_{p=0}^{kb}\binom{kb}{p}\left[1+\frac{pa}{b}\left|\xi\right|^2+\frac{pa}{2b}
 \left(\frac{pa}{b}-1\right)|\xi|^4+\frac{pa}{6b}\left(\frac{pa}{b}-1\right)\left(\frac{pa}{b}-2\right)|\xi|^6\right.\\
 & \left.+\dots +\frac{pa}{h!b}\left(\frac{pa}{b}-1\right)\cdots\left(\frac{pa}{b}-(h-1)\right)|\xi|^{2h}+\dots\right]\\
 &=\sum_{h=0}^\infty \frac{A_h}{h!}|\xi|^{2h},
\end{split}\end{equation}
where $A_0= \sum_{p=0}^{kb}\binom{kb}{p}$
 and
 \begin{equation}\label{eqahseq}
 \quad A_{h}=\sum_{p=0}^{kb} B_{h,p}, \quad B_{h,p}=\binom{kb}{p}\frac{pa}{b} \left(\frac{pa}{b}-1\right)\cdots\left(\frac{pa}{b}- (h-1)\right), \quad h>0.
 \end{equation}

Hence we have

$$
\frac{B_{\ell,p}}{B_{\ell,1}}=\binom{kb}{p}\frac{p}{bk}
\prod_{s=1}^{\ell-1}\frac{bs-ap}{bs-a}
=\binom{kb}{p}\frac{p}{bk}
\prod_{s=1}^{\ell-1}\left(1-\frac{a\left(p-1\right)}{bs-a}\right), \forall \ell.
$$
Since 
$
\sum_{s=1}^{\infty}\frac{a\left(p-1\right)}{bs-a}=+\infty,
$
we conclude that (see e.g. \cite[Section 3.7]{KnoppBook1956SequencesAndSeries})
$$
\lim_{\ell\to \infty} \frac{B_{\ell,p}}{B_{\ell,1}}=0, \qquad p=2,3,\dots,kb
$$
and therefore that 
\begin{equation}\label{eqlimaob}
\lim_{h\to \infty} \frac{A_h}{B_{h,1}}=1.
\end{equation}
Since we are assuming that $\frac{a}{b}\not\in\Z^+$, from \eqref{eqahseq} we get 
$$
B_{h+1,1}=ka\left(\frac{a}{b}-1\right)\dots\left(\frac{a}{b}-h\right)\neq 0,\  \forall h\in \Z^+.
$$
Clearly $\frac{a}{b}-h<0$, for any $h>\frac{a}{b}$. Hence, if $h_0$ is sufficiently large, from \eqref{eqlimaob}, we see that 
$\left\{A_h\right\}_{h>  h_0}$ is an alternate sequence.  We have  proved that the Taylor expansion \eqref{eqprlemp} has infinite negative coefficients when $\frac{a}{b}\not\in \Z$, as desired.
The proof is complete.
\end{proof}


\begin{remark}\rm\label{weakbCH}
Theorem \ref{mainteor} and Proposition \ref{propproj}  show that condition (b) in Question A
cannot be weakened to the condition that the metric $\alpha g^*$ is finitely projectively induced.
 \end{remark}

\begin{lem}\label{lemmamain}
Let $(M, g)$ be a \K\ manifold such that $g$ is projectively induced. Assume that there exists  $\alpha\in \Z^+$ such that $(M, \alpha g)$
admits a Fubini-Study compactification $(\tilde M, \alpha \tilde g)$. Then $(\tilde M, \tilde g)$
is a Fubini-Study compactification of $(M, g)$.
\end{lem}
\begin{proof}
By assumption there exist   holomorphic isometric embeddings  $\varphi :M\rightarrow \C P^{N}$, $J:(M, \alpha g)\rightarrow (\tilde M, \alpha\tilde g)$
and $\varphi_\alpha :\tilde M\rightarrow \C P^{s_\alpha}$ for some positive integers $N$ and $s_\alpha$. Let $V_\alpha: \C P^N\rightarrow \C P^{N_\alpha}$ be the Calabi's map (see (cf. \cite[Th.13]{CALABI1953diast})), i.e. $V_\alpha^*g_{FS}=\alpha g_{FS}$. Since without loss of generality we can assume $\varphi_\alpha$ to be full 
we can also assume $s_\alpha\leq N_\alpha$.
Then if $i: \C P^{s_\alpha}\rightarrow \C P^{N_\alpha}$ denotes  the natural totally geodesic inclusion of 
$\C P^{s_\alpha}$ into $\C P^{N_\alpha}$
it follows that  the maps  $V_\alpha\circ \varphi: M\rightarrow \C P^{N_\alpha}$ and
$i\circ {\varphi_{\alpha}}_{|M}\circ J:M\rightarrow \C P^{N_\alpha}$
are two holomorphic isometric immersions inducing the same \K\ metric
$\alpha g$. By Calabi's rigidity theorem \cite[Th.9 ]{CALABI1953diast} there exists a unitary transformation $U$ of $\C P^{N_\alpha}$ such that 
$V_\alpha\circ \varphi=U\circ i\circ {\varphi_{\alpha}}_{|M}\circ J$. 
Then the holomorphic map 
$$\tilde\varphi:=({V_\alpha}_{|V_\alpha(\C P^N)})^{-1}\circ U\circ i\circ \varphi_\alpha :\tilde M\rightarrow \C P^N$$
satisfies $\tilde\varphi^*g_{FS}=\tilde g$ and hence $(\tilde M, \tilde g)$ turns out to be a Fubini-Study compactification of $(M, g)$.
\end{proof}

The following proposition  describes  the projective inducibility of  multiples of the metric 
$g_{\Omega, \mu}$.

\begin{prop}(\cite[Th.2]{LOIZEDDA2011KEsubProjInf}\label{ioemiki})
Let $M_{\Omega, \mu}$ be a  CH domain. For  a real number $\alpha> 0$,  the \K\ metric $\alpha g_{\Omega, \mu}$ is
infinitely projectively induced  iff $(\alpha+s)\mu$ belongs to $W(\Omega)\setminus \{0\}$ for all integer $s\geq 0$.
\end{prop}

\begin{proof}[Proof  of Theorem \ref{mainteor}]
By  combining Proposition \ref{propproj},  Proposition \ref{ioemiki} and the fact that the metric $g_{\Omega, \mu}$ is complete we deduce that 
the assumptions that 
$\mu\in \Z^+$,  and  $\alpha$ is an integer  sufficiently large  
are necessary and sufficient conditions for 
$(M_{\Omega, \mu}, \alpha g_{\Omega, \mu})$
to admit  a Fubini-Study completion (cf. Remark \ref{rmkcompl}),  and the dual \K\ metric 
$\alpha g_{\Omega, \mu}^*$  to be finitely projectively induced.

Notice  that if $M_{\Omega, \mu}$ is homogeneous then  $\Omega =\C H^n$, $\mu=1$ and hence $(M_{\Omega, \mu}=\C H^{n+1}, g_{\Omega, \mu}=g_{hyp})$
which admits a Fubini-Study compactification given by $(\C P^{n+1}, g_{FS})$. 
Hence, it remains to  prove that if one assumes 
that  $(M_{\Omega, \mu}, \alpha g^*_{\Omega, \mu})$ admits a Fubini-Study compactification then  $M_{\Omega, \mu}$ is homogeneous, i.e. $\Omega_c=\C P^n$.
By Lemma \ref{lemmamain} we can assume $\alpha =1$ and hence  there exists  a holomorphic isometry  $\Psi$  from $\C^{n+1}$ into a finite dimensional complex projective space inducing $g^*_{\Omega, \mu}$. Assume by contradiction that $\Omega_c\neq \C P^n$.
By Calabi's rigidity, up to a unitary transformations of the ambient projective space we can assume 
that $\Psi=\tilde {F}_{\mu}\circ J:\C^{n+1}\rightarrow \C P^{d_{\mu}+1}$ is given by the embedding \eqref{eqembdcarp}. 
Let now 
\begin{equation}\label{equazioniP}
P_j(X_0, \dots, X_{d_\mu}) = 0, \ \ j= 1, \dots, k, \ \deg P_j\geq 2,
\end{equation}
be the homogeneous polynomial equations which define the image $F_\mu(\Omega_c)$ of the embedding $F_\mu: \Omega_c\rightarrow \C P^{d_\mu}$ 
given by \eqref{eqembdcarmeno} (the condition  $\deg P_j\geq 2$ for all $j=1, \dots ,k$ comes from the fact that $F_\mu$ is a full embedding and $\Omega_c\neq \C P^n$).

We claim that
\begin{equation}\label{chiusura}
(\overline{\tilde F_{\mu}\circ J)(\C^{n+1})} = K,
\end{equation}
where
$$K:= \{ [X_0, \dots, X_{d_\mu}, X_{d_\mu+1}] \ | \ P_j(X_0, \dots, X_{d_\mu}) = 0, \ \ j= 1, \dots, k \  \}.$$

The inclusion  $(\overline{\tilde F_{\mu}\circ J)(\C^{n+1})} \subseteq K$ is immediate by construction.
In order to prove that $K\subseteq (\overline{\tilde F_{\mu}\circ J)(\C^{n+1})}$
let $Q =\left[X_0, \dots, X_{d_\mu}, X_{d_\mu+1}\right] \in K$.

We have two cases: on the one hand, if $X_0 = \cdots = X_{d_\mu} = 0$ (i.e. $Q = [0, \dots, 0, 1]$), then $Q \in (\overline{\tilde F_{\mu}\circ J)(\C^{n+1})}$ since it is the limit of any sequence 
$$(\tilde F_{\mu}\circ J) (z, w_j)=[s_0^{(\mu)}(J(z)), \dots, s_{d_\mu}^{(\mu)}(J(z)), w_j s_0^{(\mu)}(J(z)) ]$$ 
with $w_j \rightarrow + \infty $ and any (fixed) $z\in \C^{n+1}$.

On the other hand, let $(X_0, \dots, X_{d_\mu}) \neq (0, \dots, 0)$, say $X_k\neq 0$. 
Then $[X_0, \dots, X_{d_\mu}] \in F_\mu(\Omega_c)\subset\C P^{d_\mu}$, i.e. $\left[X_0, \dots, X_{d_\mu}\right] =\left[s^{(\mu)}_0(p), \cdots  ,s^{(\mu)}_{N_\mu}(p)\right]$ for some $p \in \Omega_c$. Since $J(\C^n)$ is dense in $\Omega_c$, there exists a sequence $z_j\in \C^{n}$ such that $p_j:=J(z_j) \rightarrow p$ and then  the sequence  
$$\tilde F_\mu \left(p_j, \frac{X_{d_\mu+1}}{s_k^{(\mu)}(p)}\right)=\left[s_0^{(\mu)}(p_j), \dots, s_{N_\mu}^{(\mu)}(p_j), s_k^{(\mu)}(p_j)\frac{X_{d_\mu+1}}{s_k^{(\mu)}(p)}\right]\rightarrow Q.$$ 
This shows that $Q=[X_0, \dots, X_{d_\mu}, X_{d_\mu+1}] \in (\overline{\tilde F_\mu\circ J)(\C^{n+1})}$ and proves the claim (\ref{chiusura}).

In light of the above, in order to end the proof  it will suffice to show that $K$ is not smooth unless $\Omega = \C H^n$ and $\mu = 1$.

In order to do that, notice that $[0, \dots, 0, 1] \in K$ and that, by the implicit function theorem, in order for the equations (\ref{equazioniP}) to define a smooth submanifold 
of $\C P^{d_{\mu}+1}$ in a neighbourhood of $(X_0, \dots, X_{d_\mu+1}) = (0, \dots, 0, 1)$ at least one of  the homogeneous polynomials $P_j$ must be linear in constrast with $\deg P_j\geq 2$,  for all $j=1, \dots ,k$.
\end{proof}

\section{The \K\ dual of $(M_{\Omega, \mu}, \hat g_{\Omega, \mu})$ and the proof of Theorem \ref{mainteor2}}\label{secproof2}
In the proof of Theorem \ref{mainteor2}
we need the following  proposition 
interesting on its  own sake.
\begin{prop}\label{propmod}
 Let $M_{\Omega, \mu}$ be a  CH domain. 
 Then the \K\ metric $\hat g_{\Omega, \mu}$
whose associated \K\ form is given by \eqref{modmetric}  is complete.
Futhermore, 
for  a real number $\alpha> 0$,  the \K\ metric 
$\alpha\hat g_{\Omega, \mu}$ is 
infinitely projectively induced  iff $(2\alpha+s)\mu$ belongs to $W(\Omega)\setminus \{0\}$ for all integer $s\geq 0$.
\end{prop}

The second part of the  proposition can be considered the  analogous of Proposition \ref{ioemiki} above for the metric $g_{\Omega, \mu}$. Indeed its  proof is an adaptation of the proof given in \cite[Th.2]{LOIZEDDA2011KEsubProjInf}.

\begin{proof}[Proof  of Proposition \ref{propmod}]
Let $\gamma(t) = (z_1(t), \dots, z_n(t), w(t))$ be a curve in $M_{\Omega, \mu}$.
Then, it is easily seen  that 
\begin{equation}\label{normcurve}
\| \dot \gamma(t) \|_{\hat g_{\Omega, \mu}}^2 = \| \dot \gamma(t) \|_{g_{\Omega, 	\mu}}^2 + \| \dot{p(\gamma)}(t) \|_{\mu g_\Omega}^2
\end{equation}

where $p(\gamma)(t) = (z_1(t), \dots, z_n(t))$ denotes the projection of $\gamma$ on the base $\Omega$  equipped with the  metric $\mu g_{\Omega}$.

We deduce that the length of $\gamma$ with respect to $\hat g_{\Omega, \mu}$ is

\begin{equation}\label{comparIntegrals}
\int\| \dot \gamma(t) \|_{\hat g_{\Omega, \mu}}dt =\int
\sqrt{\| \dot \gamma(t) \|_{g_{\Omega, \mu}}^2 + \| \dot{p(\gamma)}(t) \|_{\mu g_\Omega}^2}dt\geq \int\| \dot \gamma(t) \|_{g_{\Omega, \mu}}dt 
\end{equation}

Now, we know (see e.g. \cite{DOCARMO1992}) that a Riemannian metric is complete if and only if the length of every {\em divergent curve} (this means that is gets out of every compact in $M_{\Omega, \mu}$) is not finite.
Let then $\gamma(t)$ be a divergent curve in $M_{\Omega, \mu}$.
Since the metric $g_{\Omega, \mu}$ is complete, then $\int \| \dot \gamma(t) \|_{g_{\Omega,\mu}} \ dt = \infty$. But then, from (\ref{comparIntegrals}) we deduce that also $\int \| \dot \gamma(t) \|_{\hat g_{\Omega, \mu}} \ dt = \infty$  and by the above criterium we deduce that also $\hat g_{\Omega,\mu}$ is complete.

In order to prove the second part of the proposition, notice that
a potential of the multiple $\alpha \hat g_{\Omega, \mu}$ is 

\begin{equation}\label{potentialmodified}
D = \log [N_\Omega^{- \mu \alpha} (N_\Omega^{\mu} - |w|^2)^{-\alpha} ].
\end{equation}
 In fact, it is easily seen that (\ref{potentialmodified}) is indeed  the diastasis centred at the origin
for the metric $\alpha s\hat g_{\Omega,\mu}$, i.e. $D=D_0^{\alpha\hat g_{\Omega,\mu}}$.
 In order to apply Calabi's criterium  (see \cite[Th.9]{CALABI1953diast}) 
 consider the expansion

\begin{equation}\label{CalabiExp}
e^D - 1 = \sum_{j ,k=0}^{\infty} B_{jk} (zw)^{m_j} (\bar z \bar w )^{m_k}, 
\end{equation}
where  $m_j = (m_{j,1}, \dots, m_{j, n}, m_{j, n+1})$ is a multiindex, $(zw)^{m_j} = z_1^{m_{j,1}} \cdots z_n^{m_{j, n}} w^{m_{j, n+1}}$ and the multiindices $m_j$ are ordered so that their norms $|m_j| = m_{j,1} +  \cdots + m_{j, n} + m_{j, n+1}$ satisfy $|m_j| \leq |m_{j+1}|$ and by using the lexicographic order when $|m_j| = |m_k|$. The coefficient $B_{jk}$ in (\ref{CalabiExp}) is clearly the partial derivative 

$$B_{jk} = \frac{1}{m_j! m_k!} \frac{\partial^{|m_j| + |m_k|}}{\partial(zw)^{m_j} (\bar z \bar w)^{m_k}} e^D$$
evaluated at the origin.

Now, by following the same outline of the proof in \cite{LOIZEDDA2011KEsubProjInf}, we first notice that $B_{jk}=0$ when the \lq\lq truncated norms'' $|m_j|_n = m_{j,1} +  \cdots + m_{j, n}$ and $|m_k|_n = m_{k,1} +  \cdots + m_{k, n}$ are different or when $m_{j, n+1} \neq m_{k, n+1}$.

Indeed, this follows from the fact that the map $z \mapsto e^{i \theta} z$ is an automorphism of the base domain $\Omega$. Then, since the diastasis $\log N_\Omega(z, \bar z)$ is invariant by automorphisms, we have $N_\Omega(e^{i \theta} z, e^{- i \theta} \bar z) = N_\Omega(z, \bar z)$. Similarly, since the diastasis (\ref{potentialmodified}) depends radially on $w$, also the map $w \mapsto e^{i \phi}w$ is an isometry.

Then, if in the expansion (\ref{CalabiExp}) of $e^D -1 = N_\Omega(z, \bar z)^{- \mu \alpha} [N_\Omega(z, \bar z)^{\mu} - |w|^2)]^{-\alpha}$  there is a monomial $(zw)^{m_j}(\bar z \bar w)^{m_k}$ (with $B_{jk} \neq 0$) we must have

$$(zw)^{m_j}(\bar z \bar w)^{m_k} = e^{i \theta (|m_j|_n - |m_k|_n)} e^{i \phi (m_{j, n+1} - m_{k, n+1})} (zw)^{m_j}(\bar z \bar w)^{m_k}$$

for every $\theta, \phi$, which implies the claim.

Since $B_{jk}\neq 0$ implies $m_{j,1} +  \cdots + m_{j, n} =  m_{k,1} +  \cdots + m_{k, n}$ and $m_{j, n+1} = m_{k, n+1}$, it follows that  $B_{jk}\neq 0$ implies $|m_j| = |m_k|$. Then, by the ordering chosen for the $m_j$'s, we deduce that the matrix $B = (B_{jk})$ has the following diagonal block structure:

$$
A = \left( \begin{array}{cccccc}
0 & 0 & 0& 0 & 0 & \\
0 & E_1 & 0 & 0 & 0 & \\
0 & 0 & E_2 & 0 & 0 & \\
0 & 0 & 0 & E_3 & 0 & \vdots \\
0 & 0 & 0 & 0 & E_4 & \\
 & & \dots & & & 
\end{array} \right)$$
where  $E_i$ contains the entries corresponding to derivatives with respect to multiindices $m_j$ having norm $|m_j| = i$.

In turn, up to rearranging the order of the $m_j$'s having norm $i$, we can assume that every block $E_i$ has the following diagonal block structure\footnote{Rearranging the order has the effect to apply to $E_i$ a permutation both to the rows and the columns, which in turn can be obtained by replacing $E_i$ with a congruent matrix ${}^T P E_i P$ (where $P$ is a permutation matrix), which is permitted since $E_i$ is positive definite if and only if ${}^T P E_i P$ is.} 

$$
E_i = \left( \begin{array}{ccccccc}
F_{z(i), w(0)} & 0 & 0& 0 & 0 & & 0\\
0 & F_{z(0), w(i)} & 0 & 0 & 0 & & 0\\
0 & 0 & F_{z(i-1), w(1)} & 0 & 0 & & 0\\
0 & 0 & 0 & F_{z(i-2), w(2)} & 0 & \dots & 0\\
 & & \dots & & & & 0\\
0 & 0 & 0 & 0 & 0 & & F_{z(1), w(i-1)}  \\
\end{array} \right)$$

where we are denoting by $F_{z(i-s), w(s)}$ ($s=0, \dots, i$) the block which  contains the entries $B_{jk}$ corresponding to the derivatives with respect to multiindices such that $|m_j|_n = |m_k|_n = i-s$ and $m_{j, n+1} = m_{k, n+1} = s$ i.e., in other words, such that there are $i-s$ derivatives with respect to $z$ and $\bar z$ and $s$ derivatives with respect to $w$ and $\bar w$ (recall that in order for $B_{jk}$ to be non zero we must have $|m_j|_n = |m_k|_n$ and $m_{j, n+1} = m_{k, n+1}$).

In order to prove that the matrix is positive definite, we just need to show that each of these blocks $F_{z(i-s), w(s)}$ (for any $i=1, 2, \dots$ and any $s=0, \dots, i$) is positive definite.

Let us begin from the block $F_{z(i), w(0)}$,  By definition, its entries are obtained by deriving the function $N_\Omega^{- \mu \alpha} (N_\Omega^{\mu} - |w|^2)^{-\alpha}$ $i$ times with respect to the $z$'s and $i$ times with respect to the $\bar z$'s (with no derivatives with respect to $w$ or $\bar w$). Since the derivatives are evaluated at the origin, one gets the same by applying the same derivatives to the function

$$N_\Omega^{- \mu\alpha} (N_\Omega^{\mu} )^{-\alpha}= N_\Omega^{-2 \alpha \mu} = e^{2 \alpha\mu D_0^{g_\Omega}}$$

where  $D_0^{g_\Omega}$
denotes the  diastasis function of $g_\Omega$ (cf. \eqref{defomom} above).
By Calabi's criterium and Theorem A  one gets that the blocks under consideration are positive definite if and only if 
\begin{equation}\label{condiz1}
2 \alpha\mu \in W(\Omega)\setminus \{0\}.
\end{equation}
Let us now consider the block $F_{z(0), w(i)}$. by definition, its entries are obtained by deriving the function $N_\Omega^{- \mu \alpha} (N_\Omega^{\mu} - |w|^2)^{-\alpha}$ $i$ times with respect to $w$ and $i$ times with respect to $\bar w$ (with no derivatives with respect to the $z$'s or the $\bar z$'s): since the derivatives are evaluated at the origin, by using the fact that $N_\Omega(0,0)=1$ one gets the same by applying the same derivatives to the function

$$(1 - |w|^2)^{-\alpha} = \left( \sum_{j=0}^{\infty} |w_j|^2 \right)^\alpha,$$

which shows that the block $F_{z(0), w(i)}$ is positive definite.

Finally, let us consider the block $F_{z(i-s), w(s)}$, which by definition consists of the entries  obtained by deriving the function $N_\Omega^{- \alpha\mu} (N_\Omega^{\mu} - |w|^2)^{-\alpha}$ $i-s$ times with respect to the $z$'s, $i-s$ times with respect to the $\bar z$'s, $s$ times with respect to $w$ and $s$ times with respect to $\bar w$.

We first notice that

\begin{equation}\label{induct}
\frac{\partial^{2s}}{\partial w^s \partial \bar w^s}|_{w=0} N_\Omega^{- \alpha\mu} (N_\Omega^{\mu} - |w|^2)^{-\alpha} = N_\Omega^{- \alpha\mu}  s! \alpha(\alpha+1) \cdots (\alpha + s-1) (N_\Omega^{\mu} - |w|^2)^{-\alpha - s}
\end{equation}

Indeed, this immediately follows from the general formula\footnote{One immediately sees this formula by comparing the Taylor expansions $f(x) = \sum_{k=0}^{\infty} \frac{f^{(k)(0)}}{k!} x^k$ (with $x = |w|^2$) and $f(w, \bar w) = \sum_{i, j=0}^{\infty} \frac{1}{i! j!} \frac{\partial^{i+j}}{\partial w^i \partial \bar w^j}|_{w=0} w^i \bar w^j$.} for radial functions $\frac{\partial^{2s}}{\partial w^s \partial \bar w^s}|_{w=0} f(|w|^2) = s! f^{(s)}(0)$ and the fact that 
$$\frac{d^s}{dx^s} (A - x)^{-\alpha} = \alpha(\alpha+1) \cdots (\alpha+s-1) (A - x)^{-\alpha-s}.$$

Then, since the derivatives are evaluated at the origin, we deduce that the entries of $F_{z(i-s), w(s)}$ can be equivalently be obtained by deriving the function 

$$N_\Omega^{- \alpha\mu }  s! \alpha(\alpha+1) \cdots (\alpha + s-1) (N_\Omega^{\mu} )^{-\alpha - s} = s! \alpha(\alpha+1) \cdots (\alpha + s-1) (N_\Omega^{\mu} )^{-2\alpha - s} $$

$i-s$ times with respect to the $z$'s and $i-s$ times with respect to the $\bar z$'s.

Since 
$$(N_\Omega^{\mu} )^{-2\alpha - s} = e^{(2\alpha + s)\mu D_0^{g_\Omega}}.$$ 
by applying again  Calabi's criterium, by Theorem A and taking into account also (\ref{condiz1}) we finally deduce that  $\alpha \hat g_{\Omega, \mu}$
is projectively induced if and only if 
$(2\alpha + s)\mu \in W(\Omega)\setminus \{0\}$
for every integer $s \geq 0$. 
This concludes the proof of the proposition.
\end{proof}


We are now in a position to prove Theorem \ref{mainteor2}.

\begin{proof}[Proof  of Theorem \ref{mainteor2}]
It is easily seen that  Calabi's diastasis function 
$D_0^{\hat g_{\Omega, \mu}}$ for the \K\  metric $\hat g_{\Omega, \mu}$ 
 (cf. \eqref{modmetric}) on the CH domain 
$M_{\Omega, \mu}\subset \C^{n+1}$ is given by 
$$D_0^{\hat g_{\Omega, \mu}}(z, \bar z)=-\log (N_\Omega(z, \bar z)^{\mu} - |w|^2) -\log N_\Omega^\mu(z, \bar z)$$
and that it   cannot be extended to an open subset of $\C^{n+1}$ strictly containing $M_{\Omega, \mu}$.
Thus,  the CH domain $(M_{\Omega, \mu}, \hat g_{\Omega, \mu})$
is an exact domain (cf. Definition \ref{exactdomain}).
Moreover, the dual  $(1, 1)$-form 
\begin{equation}\label{modmetricdual}
\hat \omega^*_{\Omega, \mu} =\frac{i}{2\pi} \partial \bar\partial \log (N_\Omega(z, -\bar z)^{\mu} + |w|^2) + \frac{i}{2\pi} \partial \bar\partial \log N_\Omega^\mu(z, -\bar z) 
\end{equation}
obtained by \eqref{modmetric} with the \lq\lq dual trick'' \eqref{diastasisstar} in Definition \ref{dualcomplexdom} turns out to be a \K\ form on the whole $\C^{n+1}$ and if $\hat g^*_{\Omega, \mu}$ denotes the associated \K\ metric and $D_0^{\hat g^*_{\Omega, \mu}}$ its Calabi's diastasis functions, one has
\begin{equation}\label{diastmodmetricdual}
D_0^{\hat g^*_{\Omega, \mu}}(z, \bar z)=
\log (N_\Omega(z, -\bar z)^{\mu} + |w|^2) + \log N_\Omega^\mu(z, -\bar z).
\end{equation}
Thus $(\C^{n+1}, \hat g^*_{\Omega, \mu})$ is the \K\ dual of $(M_{\Omega, \mu}, \hat g_{\Omega, \mu})$.
We are going to show that for all $\alpha, \mu\in \Z^+$,  the pair $(\C^{n+1}, \alpha\hat g^*_{\Omega, \mu})$ has  a Fubini-Study compactification.
Specifically, for  $\mu\in \Z^+$ we will  construct 
a compact \K\ manifold $(P_{\Omega_c, \mu}, g_{P, \mu})$, where $g_{P, \mu}$ is  (finitely) projectively induced, along with 
a holomorphic  embedding with dense image $\hat J:\C^{n+1}\rightarrow P_{\Omega_c, \mu}$ such that $\hat J^*g_{P, \mu}=\hat g^*_{\Omega, \mu}$. 
This will complete  the proof of Theorem \ref{mainteor2}. Indeed by choosing $\mu\in \Z^+$ and $\alpha\in\Z^+$ sufficiently large,
Proposition \ref{propmod}  and the structure of  $W(\Omega)$ ensure that   $(M_{\Omega, \mu}, \alpha\hat g_{\Omega, \mu})$ has a Fubini-Study completion  
(cf. Remark \ref{rmkcompl}). Furthermore,  $\alpha g_{P, \mu}$
will be projectively induced using Calabi's map (see \cite[Th.13]{CALABI1953diast}), and $(P_{\Omega_c, \mu}, \alpha g_{P, \mu})$ will serve as  the Fubini-Study compactification of $(\C^{n+1}, \alpha\hat g^*_{\Omega, \mu})$.

In order to do that we  simplify the notation as follows.
Fix $\mu >0$ and set  $F=F_\mu$ for the map given by 
\eqref{eqembdcarmeno}.  We also  set   $N=d_\mu$, 
and we denote by $F_k, k=0, \dots ,N$, the sections  $s_k^{(\mu)}$ appearing as the components of  the map $F_\mu$.

Ultimately, we have an embedding
\begin{equation}\label{newF}
F:\Omega_c\rightarrow \C P^N, p\mapsto [F_0(p), \dots , F_N(p)]
\end{equation}
such that $F^*g_{FS}=\mu g_{\Omega_c}$ and $F_0(J(z))\neq 0$, for all $z\in \C^n$, where $J: \C^n\rightarrow \Omega_c$ is the embedding given by \eqref{embeddings}.
Moreover,  if  $U_k = \{ [X_0, \dots, X_N] \in \C P^N \ | \ X_k\neq 0 \}$ are  the open affine subsets we have 
\begin{equation}\label{densityagain}
\Omega_c=J(\C^n)\sqcup F^{-1}(\{X_0=0\}).
\end{equation}

Notice also that by 
$${\mu \omega_{\Omega_c}}_{|F^{-1}(U_0)}=F^*\omega_{FS}=\frac{i}{2\pi}\partial\bar\partial\log(1+\sum_{k=1}^{N}|\frac{F_k(p)}{F_0(p)}|^2)=\frac{i}{2\pi}\partial\bar\partial\log
\frac{\|F(p)\|^2}{|F_0(p)|^2}$$ 
and 
$$J^*(\mu\omega_{\Omega_c})=
\frac{i}{2\pi}\partial\bar\partial\log N_\Omega^\mu (z, -\bar z)$$
we deduce that 
\begin{equation}\label{Ngood}
N_\Omega^\mu(z, -\bar z)=\frac{\|F(J(z))\|^2}{|F_0(J(z))|^2}.
\end{equation}

Let now $P(\C \oplus O(1))$ be  the $\C P^1$-bundle over $\C P^N\times \C P^N$, where
$O(1) \rightarrow \C P^N$ is  the hyperplane bundle. Denote by $Q_{\C P^N}=\Delta_{\C P^N}^*P(\C \oplus O(1))$ the  pull-back bundle over  $\C P^N$, where 
$\Delta_{\C P^N}$
is the diagonal map of $\C P^N\rightarrow \C P^N\times \C P^N$.
Analogously let 
$$P_{\Omega_c,\mu}=\Delta_{\Omega_c}^*P(\C \oplus L^{\otimes\mu})$$
 be  the $\C P^1$-bundle over $\Omega_c$ 
where  $P(\C \oplus L^{\otimes\mu})$ is the $\C P^1$-bundle over $\Omega_c\times \Omega_c$, $L^{\otimes\mu} = F^*(O(1))$ and $\Delta_{\Omega_c}$
is the diagonal map of $\Omega_c\rightarrow \Omega_c\times \Omega_c$.

Then we have a natural holomorphic embedding with dense image of $\C^{n+1}$  into $P_{\Omega_c, \mu}$   given by 
\begin{equation}\label{embCinP}
\hat J: \C^{n+1} \rightarrow F^{-1}(U_0) \times \C P^1 \subseteq P_{\Omega_c, \mu}, \ \ (z, w) \mapsto (J(z), [1, w]).
\end{equation}

Consider  the biholomorphic map $\Psi$ from $Q_{\C P^N}$ into the blowup  
$\Bl_{[0, \dots, 0,  1]}(\C P^{N+1})$ of $\C P^{N+1}$ at the point $[0, \dots, 0, 1]$
given locally on the trivializing open subset $U_k \times \C P^1 \subseteq Q_{\C P^N}$, $k=0, \dots ,N$, by 
\begin{equation}\label{identif}\begin{split}
\Psi_{|U_k\times \C P^1}: U_k \times \C P^1 \rightarrow \Bl_{[0, \dots, 0, 1]}(\C P^{N+1})\subset \C P^{N+1}\times \C P^N,\\
([Z_0, \dots, Z_N], [W_0, W_1]) \mapsto ([Z_0 W_0, \dots, Z_N W_0, Z_k W_1], [Z_0, \dots, Z_N]).
\end{split}
\end{equation}
It  is easily seen (see e.g. \cite[Chapter V]{HARTSHORNE1977})  that the maps $\Psi_{|U_k \times \C P^1}$ glue to a well-defined biholomorphism  $\Psi : Q_{\C P^N}\rightarrow \Bl_{[0,\dots,0, 1]}(\C P^{N+1})$).
Recall that 
\begin{equation}\label{defBU}
\Bl_{[0,\dots ,s0, 1]}(\C P^{N+1}) = \left\{\left([X_0, \dots, X_N, X_{N+1}], [Y_0, \dots, Y_N]\right) \ | \ X_i Y_j = X_j Y_i, \ i,j=0, \dots, N \ \right\}.
\end{equation}
 Since the latter is a submanifold of the product $\C P^{N+1} \times \C P^N$, it can be  embedded into a complex projective space by using the restriction of the Segre embeddings:
$$\Bl_{[0,\dots ,0, 1]}(\C P^{N+1})\subset\C P^{N+1} \times \C P^N \stackrel{Segre}{\rightarrow}\C P^K,\ K=\frac{(N+1)(N+4)}{2}$$
Thus, one can    consider the holomorphic embedding
$$\varphi=Segre\circ\Psi\circ\varphi_{P_{\Omega_c, \mu}}:P_{\Omega_c, \mu}\rightarrow \C P^{K},$$
where $\varphi_{P_{\Omega_c, \mu}}:P_{\Omega_c, \mu}\rightarrow Q_{\C P^N}$ is given locally by 
\begin{equation}\label{defphiP}
F^{-1}(U_k)\times \C P^1\rightarrow U_k\times \C P^1, \  (p, [W_0, W_1])\mapsto (F(p), [W_0, W_1]).
\end{equation}

Thus,  for  $\mu\in\Z^+$ we define the desired projectively induced  metric $g_{P, \mu}$
on $P_{\Omega_c, \mu}$ by 
\begin{equation}\label{defgP}
g_{P,\mu}:=\varphi^*g_{FS}.
\end{equation}
It remains then  to show  that $\hat J^*g_{P, \mu}=\hat g^*_{\Omega, \mu}$
(and hence  $(P_{\Omega_c, \mu}, g_{P, \mu})$ will be the desired  Fubini-Study compactification of $(\C ^{n+1}, \hat g^*_{\Omega, \mu})$, with $\hat J:\C^{n+1}\rightarrow(P_{\Omega_c, \mu}, g_{P, \mu})$).
In order to do this, notice that by \eqref{identif} the embedding 
$$\Psi\circ\varphi_{P_{\Omega_c, \mu}} :P_{\Omega_c, \mu}\rightarrow \Bl_{[0,\dots 0, 1]}(\C P^{N+1})\subset \C P^{N+1} \times \C P^N$$
reads on $F^{-1}(U_k)\times\C P^1$
as
\begin{equation}\label{locemb}
{\Psi\circ\varphi_{P_{\Omega_c, \mu}}}_{|F^{-1}(U_k)\times \C P^1}:(p, [W_0, W_1])\mapsto ([F(p)W_0, F_k(p)W_1], [F(p)]).
\end{equation}
Since the pull-back of the Fubini-Study metric on 
$\C P^K$ via the Segre embedding is given by the \K\  product of the Fubini-Study metrics
on $\C P^{N+1}$ and $\C P^N$, one gets
that  the \K\ form $\omega_{P,\mu}$ associated to $g_{P, \mu}$
is given on 
$F^{-1}(U_k)\times\C P^1$ by:
\begin{equation}\label{omegaP}
{{\omega_{P, \mu}}}_{|F^{-1}(U_k)\times\C P^1}=
\frac{i}{2\pi}\partial\bar\partial\log\left(\frac{\|F(p)\|^2}{|F_k(p)|^2}|W_0|^2+|W_1|^2\right)+
\frac{i}{2\pi}\partial\bar\partial\log \frac{\|F(p)\|^2}{|F_k(p)|^2}.
\end{equation}
Notice that 
$$\hat J (\C^{n+1})\subset F^{-1}(U_0)\times \{W_0\neq 0\},\  \varphi_{P_{\Omega_c, \mu}}(F^{-1}(U_0)\times \{W_0\neq 0\})\subset U_0\times \{W_0\neq 0\}$$
$$\Psi (U_0\times \{W_0\neq 0\})\subset \{X_0\neq 0\}\times \{Y_0\neq 0\}, \ Segre (\{X_0\neq 0\}\times \{Y_0\neq 0\})\subset \{T_0\neq 0\}.$$
where  $[T_0, \dots , T_K]$
denote  the  homogeneous coordinates of  $\C P^K$.
Hence, by passing to affine coordinates $(\frac{X_1}{X_0}, \dots , \frac{X_{N+1}}{X_0})$, $(\frac{Y_1}{Y_0}, \dots , \frac{Y_{N}}{Y_0})$, $(\frac{T_1}{T_0}, \dots , \frac{T_{K}}{T_0})$
and by using  \eqref{omegaP} (with $k=0$) one gets
$${\omega_{P,\mu}}_{|F^{-1}(U_0)\times \{W_0\neq 0\}}=
\frac{i}{2\pi}\partial\bar\partial\log \left(\frac{\|F(p)\|^2}{|F_0(p)|^2}+|w|^2\right)+\frac{i}{2\pi}\partial\bar\partial\log \frac{\|F(p)\|^2}{|F_0(p)|^2},$$

Then \eqref{Ngood} and  \eqref{embCinP} yield
$$\hat J^*{\omega_{P, \mu}}=\frac{i}{2\pi}\partial\bar\partial\log\left (\frac{\|F(J(z))\|^2}{|F_0(J(z))|^2}+|w|^2\right)+\frac{i}{2\pi}\partial\bar\partial\log
\left(\frac{\|F(J(z))\|^2}{|F_0(J(z))|^2}\right)$$
$$=\frac{i}{2\pi}\partial\bar\partial\log\left (N_\Omega(z, -\bar z)^{\mu}+|w|^2\right)+\frac{i}{2\pi}\partial\bar\partial\log
\left(N_\Omega(z, -\bar z)^{\mu}\right)=\hat\omega^*_{\Omega, \mu}.$$
 This concludes the proof of the theorem.
\end{proof}
\begin{remark}\rm
With the notation introduced in the proof  of Theorem \ref{mainteor2} one has 
$$P_{\Omega_c, \mu}=\hat J(\C^{n+1})\sqcup \varphi^{-1}(\{T_0=0\})$$
in  analogy with
\eqref{densityagain} for $\Omega_c$. 
\end{remark}
One can wonder if there exist any non  integer values of $\alpha$ e $\mu$
such that $\alpha \hat g^*_{\Omega, \mu}$ is projectively induced (in analogy with  Proposition \ref{propproj} for the metric $\alpha g^*_{\Omega, \mu}$).
This cannot happen.
\begin{prop}\label{propalphamu}
$\alpha \hat g^*_{\Omega, \mu}$ is projectively induced iff $\alpha, \mu\in\Z^+$.
\end{prop}
\begin{proof}
If $\alpha, \mu\in\Z^+$ then the proof of Theorem \ref{mainteor2}
shows that $\alpha g_{P, \mu}$ and hence $\alpha \hat g^*_{\Omega, \mu}$
is finitely projectively induced.
On the other hand, 
 if  $\alpha \hat g^*_{\Omega, \mu}$ is projectively induced then by taking $z =0$ in \eqref{diastmodmetricdual} we see that 
$$\alpha\frac{i}{2\pi}\partial\bar\partial\log(1+|w|^2)={\alpha g_{FS}}_{|\C},$$
where $\C=U_{0}=\{W_0\neq 0\}$ with affine coordinate $w=\frac{W_1}{W_0}$,  
is projectively induced,  forcing $\alpha\in\Z^+$. 
It remains to prove that if $\alpha \hat g^*_{\Omega, \mu}$ is projectively induced then 
$\alpha, \mu\in\Z^+$.
As  in the proof  of Proposition \ref{propproj}, namely by first  restricting to 
$\left\{\left(\xi,w\right)\in\C^2\mid \xi=0\right\}$ and  $\left\{\left(\xi,w\right)\in\C^2\mid w=0\right\}$,
and then to 
$\left\{\left(\xi,w\right)\in\C^2\mid w=1\right\}$, 
one deduces that if   $\alpha \hat g^*_{\Omega, \mu}$ is projectively induced then 
the same holds true for 
the \K\ metric $g$ whose Calabi's diastasis function at the origin is given by 
$$D(\xi)=bk\log\left(\left(\left(1+|\xi|^2\right)^\frac{a}{2b}+1\right)\left(1+|\xi|^2\right)^\frac{a}{2b}\right)-bk\log 2$$
for some $a,b,k\in \Z^+$ and  $a,b$ coprime.
Following  the exact  same line of reasoning as in the proof of Proposition \ref{propproj},  one can  deduce 
that the condition  $\frac{a}{2b}\not\in \Z$
implies that at least  one of  the coefficients of the Taylor expansion  in $\xi$ and $\bar \xi$ around the origin
of $e^D-1$  is negative. Therefore,  the conclusion follows by Calabi's criterium.  
\end{proof}


By combining Proposition \ref{propalphamu} and Theorem \ref{mainteor2} we immediately get
\begin{cor}
$(\C^{n+1}, \alpha \hat g^*_{\Omega, \mu})$ has a Fubini-Study compactification iff $\alpha, \mu \in\Z^+$.
\end{cor}

\subsection{$(P_{\Omega_c, \mu}, g_{P, \mu})$ as $C1$ $G$-space}
In this subsection we show that $(P_{\Omega_c, \mu}, g_{P, \mu})$ is a
{\em cohomogeneity $1$ space}, for all $\mu\in \Z^+$ (see Proposition \ref{notrelaxhom} below).
Since $(P_{\Omega_c, \mu}, g_{P, \mu})$ is a Fubini-Study compactification of 
 the exact domain $(\C^{n+1}, \hat g^*_{\Omega, \mu})$,  which  admits a \K\ dual given by  $(M_{\Omega, \mu}, \hat g_{\Omega, \mu})$, this proposition  shows that 
 the condition of homogeneity in Theorem \ref{mainteor3} cannot be  relaxed  to cohomogeneity $1$.
Before proving this fact,  we briefly review some basic facts about cohomogeneity $1$ $G$-spaces,  referring the reader 
to  \cite{PODESTASPIRO1999}, \cite{ALEKSEEVSKYZUDDAS2017} and \cite{ALEKSEEVSKYZUDDAS2020} for more detailed information  on the subject.

Given a semisimple Lie group $G$, a \K\ manifold $(P, \omega_P)$ is called a {\em cohomogeneity $1$ $G$-space} (from now on, $C1$ $G$-space) if there exists an action by holomorphic isometries of $G$ on $P$ which has an orbit of (real) codimension 1 in $P$. Such an orbit is called a {\it regular orbit}, and the union of regular orbits is called the {\it regular set} $P_{reg}$; the orbits of higher codimension are called {\it singular orbits}.

A compact C1 $G$-space $P$ has regular orbits of the kind $G/S$ (where $S$ is the stabilizer of a regular point) and two singular orbits $G/H_0, G/H_1$. We say that the action of $G$ on $P$ is {\it ordinary} if 
\begin{equation}\label{ordinary}
\dim N_G(S) = \dim S + 1,
\end{equation}
where $N_G(S)$ is the  normalizer of $S$ in $G$. This assumption allows us to associate to the regular orbits $G/S$ a flag manifold $G/N_G(S)$. This flag manifold is called {\it the flag associated to the regular orbits}. 

Let now $\Omega_c$ be an irreducible Hermitian symmetric space of compact type (i.e. the dual of a Cartan domain $\Omega$). Write $\Omega_c=G/K$, where $G$ is a simple Lie group and  
$K$ is a  maximal  compact subgroup  of $G$.

Notice  that the $\C P^1$-bundle $P_{\Omega_c, \mu}\rightarrow \Omega_c$  can be trivialized 
on the open subsets $F^{-1}(U_j) = \{ p \in M \ | \ F_j(p) \neq 0 \}$ with the following identifications given by the transition functions
$$F^{-1}(U_j) \times \C P^1 \ni (p, \left[W_0, W_1\right]) \simeq \left(p, \left[W_0, \frac{F_j(p)}{F_k(p)} W_1\right]\right) \in F^{-1}(U_k) \times \C P^1,$$
where $F:\Omega_c\rightarrow \C P^N$
is given by \eqref{newF}. 

Define an action of $G$ on $P_{\Omega_c, \mu}$ locally  in the following way. Given 
$$(p, [W_0, W_1]) \in F^{-1}(U_j) \times \C P^1\subset P_{\Omega_c, \mu},$$ 
and $g p\in F^{-1}(U_k)$, then 

\begin{equation}\label{defaction}
g \cdot (p, \left[W_0, W_1\right]) = \left( g p, \left[W_0, \frac{F_j(p)}{\lambda_g^{-1}\left(F_k(g  p)\right)} W_1 \right] \right) \in  F^{-1}(U_k) \times \C P^1,
\end{equation}
where $\lambda_g: L^{\otimes\mu}_p \rightarrow L^{\otimes\mu}_{gp}$ denotes a lifting to the fibers of the action map $\Omega_c \rightarrow \Omega_c, p \mapsto g p$.
Checking that this local definition extends to a well-defined action on $P_{\Omega_c, \mu}$ and that it is in fact a group action are straightforward by using the transition functions, and we omit that. 

 Let us see more in detail the structure of the orbits in order to see that this is a C1 $G$-action. This is done more easily by using the fact that, as we have seen in the proof of Theorem \ref{mainteor2} (see \eqref{locemb}), the $\C P^1$-bundle $P_{\Omega_c, \mu}$ can be seen as a submanifold of the blowup $Bl_{[0,\dots,0, 1]}(\C P^{N+1})$ via the embedding whose local expression on  
$F^{-1}(U_k) \times \C P^1$ is
\begin{equation*}
(p, \left[W_0, W_1\right])\mapsto (\left[F(p)W_0, F_k(p)W_1\right], [F(p)]).
\end{equation*}
Then, by (\ref{defaction}), one can  easily see that  the action can be written  as
\begin{equation}\label{actionBLowup}
g \cdot \left(\left[\frac{F(p)}{ F_j(p)}W_0, W_1\right], [F(p)]\right) = \left(\left[\frac{\lambda_g^{-1}\left(F(gp)\right)}{ F_j(p)}W_0, W_1\right], [F(gp)]\right),
\end{equation}
for any $p \in \Omega_c$ with $F_j(p) \neq 0$.
s

Now, the  singular orbits of this action (\ref{actionBLowup}) can be easily described as follows:
\begin{enumerate}
\item[-] if $W_0 = 0$, we have $g \cdot ([0, \dots, 0 ,1], [F(p)]) = ([0, \dots, 0 ,1], [F(gp)])$
and then we have a singular orbit isomorphic to $F(\Omega_c)$;
\item[-] if $W_1 = 0$, we have $g \cdot ([F(p), 0], [F(p)]) = ([F(gp), 0], [F(gp)])$
 and we get again a singular orbit isomorphic to $F(\Omega_c)$.
\end{enumerate}
Hence, in this case $G/H_0\cong G/H_1\cong G/K$.

The following lemma provides a description of 
the orbits when 
$W_0, W_1\neq 0$.

\begin{lem}\label{proporbiteaperte}
The orbit of a point
$\left( \left[\frac{W_0}{W_1} \frac{F(p)}{F_k(p)}, 1 \right], [F(p)] \right)\in P_{\Omega_c, \mu}$ with $W_0, W_1 \neq 0$ can be identified with the subset of the sphere $S^{2N+1}(r) \subseteq \C^{N+1}$ of radius $r = \| \frac{W_0}{W_1} \frac{F(p)}{F_k(p)}  \|$ such that the image of this subset in $S^{2N+1}(r)/\{ e^{i \phi} \} = \C P^N$ is exactly $F(\Omega_c)$.
\end{lem}
\begin{proof}

Note  that by \eqref{actionBLowup} we can write 
\begin{equation}\label{azionesuregular}
g \cdot \left( \left[\frac{W_0}{W_1} \frac{F(p)}{F_k(p)}, 1 \right], [F(p)] \right) = \left( \left[\frac{W_0}{W_1} \frac{\lambda_g^{-1}(F(gp))}{F_k(p)}, 1 \right], [F(gp)] \right).
\end{equation}

Note also  that 
\begin{equation}\label{normcal}
\|\frac{F(p)}{F_k(p)} \| = \| \frac{\lambda_g^{-1}\left(F(gp)\right)}{F_k(p)}\|,
\end{equation}
for any $g \in G$ and any $p \in M$ with $F_k(p) \neq 0$.
This can be proved, for example,  by using  Calabi's rigidity theorem.

From (\ref{azionesuregular}) and  \eqref{normcal} one can define the   following map from the orbit ${\mathcal O}$ of  the point $\left( \left[\frac{W_0}{W_1} \frac{F(p)}{F_k(p)}, 1 \right], [F(p)] \right)$ to the sphere $S^{2N+1}(r) \subset \C^{N+1}$ of radius $r = \| \frac{W_0}{W_1} \frac{F(p)}{F_k(p)}  \|$:

\begin{equation}\label{fO}
f: {\mathcal O} \rightarrow S^{2N+1}(r), \ \ \ \left(\left[\frac{W_0}{W_1} \frac{\lambda_g^{-1}(F(gp))}{F_k(p)}, 1 \right], [F(gp)] \right) \mapsto \frac{W_0}{W_1} \frac{\lambda_g^{-1}(F(gp))}{F_k(p)}
\end{equation}

One immediately sees that $f$ is injective and is in fact a diffeomorphism onto its image, so that  the orbit identifies with the 
set 
$$X:=\left\{ \frac{W_0}{W_1} \frac{\lambda_g^{-1}(F(gp))}{F_k(p)} \ | \ g \in G \right\}\subset S^{2N+1}(r)$$

We have to prove that $X= p^{-1}(F(\Omega_c))$, 
where 
$$p: S^{2N+1}(r) \rightarrow S^{2N+1}(r)/\{ e^{i \phi} \} = \C P^N$$ 
denotes  the canonical projection to the quotient.

The inclusion $X\subset  p^{-1}(F(\Omega_c))$ is straightforward. 
Hence, it remains to  prove that for any $e^{i \phi} \in \C$ there exists $g \in G$ such that
\begin{equation}\label{aim}
\frac{\lambda_g^{-1}(F(gp))}{F_k(p)} \frac{W_1}{W_0} = e^{i \phi}\frac{F(p)}{F_k(p)} \frac{W_1}{W_0}.
\end{equation}

\smallskip

In order to show this, notice that, since $\Omega_c =G/K$ is a symmetric space, the isotropy group $K$ of a point $p$ is the centralizer of a one-dimensional compact torus $T$. Since $gp = p$ for every $g \in T$, the lifting of the action of $G$ to $L^{\otimes\mu}$ induces then an isomorphism $\lambda_g: L^{\otimes\mu}_p \rightarrow L^{\otimes\mu}_p$ of the fiber $L^{\otimes\mu}_p$, which yields a one-dimensional representation $\rho: T \rightarrow Aut(L^{\otimes\mu}_p) \simeq \C^*$ of the torus $T$. Being $U(1)$ the only non-trivial connected compact subgroup of $\C^*$, it must be either $\rho(T) = \{1\}$ or $\rho(T) = U(1)$. 
But it cannot be $\rho(T) = \{1\}$ since it is known (see, for example,  \cite{snowbook1989}) that every complex line bundle on $G/K$ is completely determined by the representation of $T$ on a fiber, and this representation is trivial if and only if the line bundle is the trivial bundle, which is not the case.

Then, we have $\rho(T) = U(1)$ and so  for every $e^{i \phi}$ there exists $g \in T$ such that 

$$ \frac{\lambda_g^{-1} (F(gp))}{F_k(p)} = e^{i \phi} \frac{F(p)}{F_k(p)}$$

which proves \eqref{aim} and  concludes the proof of the lemma.
\end{proof}

\begin{prop}\label{notrelaxhom}
 $P_{\Omega_c, \mu}$
 is a cohomogeneity $1$ $G$-space, the action of $G$ is ordinary and the 
 flag associated to the regular orbit of this action is $\Omega_c=G/K$.
 Moreover,  $g_{P, \mu}$
is a $G$-invariant \K\  metric.
\end{prop}
\begin{proof}
Notice that the regular orbits are the  orbits of the points such that $W_0, W_1 \neq 0$. Indeed, on the one hand we have seen in Lemma \ref{proporbiteaperte} 
that these orbits identify with $p^{-1}(F(\Omega_c))$, and then their (real) dimension is 
$$2 \dim_{\C}(F(\Omega_c)) + 1 = 2\dim_{\C}(\Omega_c) + 1.$$ 
On the other hand, since $P_{\Omega_c}$ is a $\C P^1$-bundle on $\Omega_c$ we have 
$$\dim_{\R}(P_{\Omega_c, \mu}) = 2(\dim_{\C}(\Omega_c) + 1) = 2\dim_{\C}(\Omega_c) + 2.$$ and then the real codimension of a regular orbit  is exactly $1$.

In order to prove that the action is ordinary,
let $S = \{ g \in G \ | \ g u_0 = u_0 \}$ be the stabilizer of a regular point $u_0\in f(\mathcal {O})\subset S^{2N+1}(r)$ (see \eqref{fO}). 
Notice that
$$N_G(S) = Y,$$
where
$$Y:=\{ g \in G \ | \ g u_0 = \lambda u_0, \ \textrm{for some} \ \lambda \in \C \ \}.$$

Indeed, if $g u_0 = \lambda u_0$ we have, for any $s \in S$,
$$g s g^{-1} u_0 = g s \lambda^{-1} u_0 = \lambda^{-1} g u_0 = u_0$$

and then $g \in S$. This proves that $Y \subseteq N_G(S)$. Now, clearly
$$Y = \{ g \in G \ | \ g [u_0] = [u_0] \} = K,\quad  [u_0]\in S^{2N+1}(r)/\{e^{i\phi}\},$$
where $K$ is the isotropy subgroup of the symmetric space $F(\Omega_c) \simeq \Omega_c = G/K$.

So $K \subseteq N_G(S)$: but since $\Omega_c = G/K$ is symmetric we have that $K$ is a maximal compact subgroup in $G$ and then we must in fact have $K = N_G(S)$.

Now, the fact we proved in Lemma \ref{proporbiteaperte}  that the orbit $G/S$ identifies with $p^{-1}(G/K)$, implies that  $\dim(G/S) = \dim(G/K) + 1$. This means that $\dim(K) = \dim(S) + 1$ which, combined with $K = N_G(S)$, shows that    \eqref{ordinary} holds true and hence $P_{\Omega_c, \mu}$ is an ordinary C1 space.
Finally the fact that $g_P$ is $G$-invariant follows by combining \eqref{omegaP}
and \eqref{normcal}.
\end{proof}

\section{The proof of Theorem \ref{mainteor3}}\label{FLAG}

In order to prove Theorem \ref{mainteor3} we start with some remarks on Calabi's diastasis
function and  \K\ duality.

Recall (see \cite{CALABI1953diast}) that among all the potentials  of a real analytic  \K\ metric $g$ on a complex manifold $M$,   
 Calabi's diastasis function  characterized by 
\begin{equation}\label{defdiastasis}
D_0^g(z)=\sum _{|I|, |J|\geq 0}a_{IJ}z^I\bar z^J, \ a_{J 0}=a_{0 J}=0.
\end{equation}
for any choice of  local complex coordinates  $(z)$ centred at $p$.

Now, if  $(U, g)$ is an exact domain
with associated \K\ form $\omega=\frac{i}{2\pi}\partial\bar\partial D^g_0$  which admits a \K\ dual  $(U^*, g^*)$
then,  by \eqref{diastasisstar},  $D_0^{g^*}(z, \bar z)=-D_0^g(z, -\bar z)$ turns out to be  real-valued.
Viceversa,  if  
there exists  a neighbourhood of $0$ such that 
$$\left(D_0^g\right)^*(z, \bar z):=-D_0^g(z, -\bar z)$$
is real-valued then $(U, g)$ admits a \K\ dual. 
Indeed, notice that 
$$\frac{\partial^2 (D_0^g)^*}{\partial {z_\alpha}\partial \bar{z_\beta}}(0)=\frac{\partial^2 D_0^g}{\partial {z_\alpha}\partial \bar{z_\beta}}(0)$$
and since the matrix on the right hand side is positive definite it follows that the matrix on the left hand side is positive definite 
on a neighborhood of $0$.
Thus $\omega^*=\frac{i}{2\pi}\partial\bar\partial (D_0^g)^*$  is a  \K\ form  in this neighborhood.
If $g^*$ is the \K\ metric associated to $\omega^*$ then $(U^*, g^*)$
is a \K\ dual of $(U, g)$, where $U^*$ is the maximal domain of extension of  $(D_0^g)^*$.
Indeed the later
 turns out to be the diastasis function 
at $0$ for the metric $g^*$ on $U^*$, i.e. 
$$(D_0^g)^*(z, \bar z)=D_0^{g^*}(z, \bar z).$$ 
This can be seen  by noticing that 
\begin{equation}\label{procedure}
(D_0^g)^*(z)=\sum _{|I|, |J|\geq 0}a^*_{IJ}z^I\bar z^J=\sum _{|I|, |J|\geq 0}(-1)^{|J|}a_{IJ}z^I\bar z^J
\end{equation}
and hence  $a^*_{J0}=a^*_{0J}=0$
which proves our claim by the very  definition of the Calabi's diastasis function.

More generally, we have 
$a^*_{IJ} = - a_{IJ}(-1)^{|J|}$.
Hence, the fact that $D_0^{g^*}(z, \bar z)=-D_0^g(z, -\bar z)$ 
is real valued gives  $a^*_{JI} = \overline{a^*_{IJ}}$, i.e. 
$$-a_{JI} (-1)^{|I|} = - \overline{a_{IJ}} (-1)^{|J|}$$
which combined with $a_{JI} = \overline{a_{IJ}}$ yields 
\begin{equation}\label{dualIJ}
(-1)^{|I|} = (-1)^{|J|},\  \mbox{if}\ \  a_{IJ} \neq 0.
\end{equation}
In view of \eqref{dualIJ} we give the following

\begin{defin}
A monomial $a_{IJ} z^I \bar z^J$ with $a_{IJ} \neq 0$
in the expansion of the Calabi diastasis function $D_0^g(z, \bar z)$
is called a {\em forbidden monomial of kind $(|I|, |J|)$} if  $(-1)^{|I|} \neq (-1)^{|J|}$.
\end{defin}

Thus we obtain

\begin{lem}\label{dualiffnotforb}
An exact domain $(U, g)$ admits a \K\ dual $(U^*, g^*)$  iff the expansion of 
$D_0^g(z, \bar z)$ does not contain forbidden monomials.
\end{lem}

It is known (see \cite{ALEKSEEVSKYPERELOMOV1986coord} and references therein) that invariant complex structures on a flag manifold $F=G/K$ are in one-to-one correspondence with {\it maximal closed nonsymmetric subsets} $Q$ of the set $R_M$ (see \cite[Def.5]{LOIMOSSAZUDDAS2019bochner}) of the black roots of $G/K$.
More precisely, the manifold $G/K$ endowed with the complex structure $J_Q$ corresponding to $Q$ is biholomorphic to the complex homogeneous manifold $G^{\bC}/K^{\bC} G^{Q}$, where $G^{Q} = \exp(\mathfrak{g}^{Q})$ and $\mathfrak{g}^{Q} = \sum_{\alpha \in Q} \bC E_{\alpha}$, where $E_\alpha$ is the so called {\em root vector} associated to $\alpha$. 
Since the product $G_{reg}^{\bC} = G^{-Q} K^{\bC} G^{Q}$ (where $G^{-Q} = \exp(\mathfrak{g}^{-Q})$ and $\mathfrak{g}^{-Q} = \sum_{\alpha \in -Q} \bC E_{\alpha}$) defines an open dense subset in $G^{\bC}$, its image in $G^{\bC}/K^{\bC} G^{Q}$ via the natural projection $G^{\bC} \rightarrow G^{\bC}/K^{\bC} G^{Q}$ defines an open dense subset in $G/K$, denoted $\Omega_{reg} = G_{reg}^{\bC}/K^{\bC} G^Q$. Clearly, $\Omega_{reg} \simeq G^{-Q}$.
 Then, one can define a biholomorphism by

\begin{equation}\label{complexcoordinates}
z = (z_{\alpha})_{\alpha \in -Q} \in \bC^n \mapsto \exp(Z(z)) \in G^{-Q} \simeq \Omega_{reg} \subseteq F
\end{equation}
where
\begin{equation}\label{Z(z)}
Z(z) = \sum_{\alpha \in -Q} z_{\alpha} E_{\alpha}
\end{equation}
(where $n$ is the cardinality of $Q$). We call the corresponding   system of complex coordinates $z_1, \dots ,z_n$ on $\Omega_{reg}$,  the {\em Alekseevsky–Perelomov coordinates}.
We then define a holomorphic embedding 
\begin{equation}\label{defJF}
J_F:\C^n\rightarrow F, z\mapsto \left[\exp(Z(z))\right]
\end{equation}
with dense image $J_F(\C^n)=\Omega_{reg}$.
The embedding  $J_F$ is the natural extension  of the embedding $J:\C^n\rightarrow \Omega_c$ given in \eqref{embeddings} 
for a Hermitian symmetric space of compact type $\Omega_c$
(see  \cite[Proof of Th.7.1, Chapter VIII Section 7, p.392]{HELGASONbookDiffGeom1978}).

Given any $G$-invariant \K\ metric $g_F$ on $F$, one has
$$J_F^* \omega_F = \frac{i}{2 \pi} \partial \bar \partial \sum_{j=_1}^p c_j \log (\Delta_j)$$
where $p$ is the number of the black nodes in the painted Dynkin diagram of $G/K$, $c_j > 0$, $j= 1, \dots, p$ and the $\Delta_j$'s are suitable minors of the matrix $A = {}^T \overline{\exp(Z)} \exp(Z)$, depending on the position of the black nodes in the diagram, called {\it admissible minors}.

In  order to prove Theorem \ref{mainteor3}, namely that 
if the exact domain $(\C^n, J_F^*g_F)$ admits a \K\ dual then $(F, g_F)$
is a Hermitian symmetric space of compact type,
 we  need the following 

\begin{lem}(\cite[Th.1]{LOIMOSSAZUDDAS2019bochner})\label{LMZ}
 Let $(F, g_F)$ be an irreducible flag manifold of classical type with second Betti number
$b_2(F) = p$, 
 with associated \K\ form $\omega_F$ determined by coefficients
$c_{i_1} , \dots , c_{i_p} > 0$ associated to the black nodes $\alpha_{i_1}, \dots , \alpha_{i_p}$ of its painted diagram.
Then, the Alekseevsky–Perelomov coordinates $z_1, \dots ,z_n$ on  $(\C^n, J_F^*g_F)$ are Bochner, up to rescaling, only in
the following cases:
\begin{enumerate}
\item[(i)] $p = 1$, for every $G$ and every $\omega$;
\item[(ii)] $p = 2$, $G = SU(d)$ and $c_{i_1} = c_{i_2}$ ;
\item[(iii)] $p = 2$, $G = SO(2d)$, the painted diagram of $F$ is
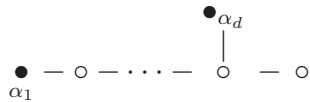
\begin{align*}
&&& \underset{\substack{\alpha_1}}{\bullet} - \underset{\substack{}}{\circ} - \cdots - \underset{\substack{}}{\overset{\overset{\textstyle\bullet_{\alpha_d}}{\textstyle\vert}}{\circ}} \,-\, \underset{\substack{}}{\circ} &&
\end{align*}
and $c_1 = 2c_d$
\end{enumerate}
\end{lem}

Let us recall that the  Bochner coordinates 
are those complex coordinates $w_1, \dots ,w_n$ in a neighbourhood of 
the origin (uniquely defined up to a unitary transformation) such that

$$D_0^{J_F^*g_F}(w, \bar w) = |w|^2 + \sum_{|J|, |K| \geq 2} b_{JK} w^J \bar w^K$$

Notice that to prove Theorem \ref{mainteor3} we can assume 
$(F, g_F)$ is irreducible. Indeed if  $F$ is not irreducible, then its painted Dynkin diagram is given by the disjoint union of (connected) painted Dynkin diagrams of simple groups, and the coordinates are  Bochner if and only if are the coordinates on each factor.

Thus Lemma \ref{LMZ} implies  that, except for the cases (i)-(iii), the expansion of  $D_0^{J_F^*g_F}(z, \bar z)$ with respect to  Alekseevsky–Perelomov coordinates  contains forbidden monomials\footnote{In fact, in the proof of \cite[Th.1]{LOIMOSSAZUDDAS2019bochner})
it is proven  that, except for cases (i)-(iii), the expansion of the diastasis always contains forbidden monomial of type $(1, 2)$.}
 of the kind $z_{i_1} \bar z_{i_2} \cdots \bar z_{i_k}$ (or their conjugates).

Then,  in order to analyze when $(\C^n, J_F^*g_F)$  can admit a \K\ dual, we just need to consider cases (i)-(iii).

In fact, in both cases (ii) and (iii), the proof of  Lemma \ref{LMZ}
(see \cite[cases 3.2.1 and 3.2.2 in Th.1]{{LOIMOSSAZUDDAS2019bochner}}) shows that the diastasis is of the kind
\begin{equation}\label{delta12}
D_0^{J_F^*g_F}= c_1 \log \Delta_1 + c_2 \log \Delta_2, \ c_1, c_2\in\R^+
\end{equation}
where both $\Delta_1$ and $\Delta_2$ contain forbidden monomials.
We are then left with considering only case (i).

When $(F, g_F)$ is not symmetric, we are going to show that the admissible minor always contains forbidden monomials of type $(2,3)$
and then the  proof of Theorem \ref{mainteor3} will  follow by Lemma \ref{dualiffnotforb}. In order to do that, we pause to prove   the following

\begin{lem}\label{lemmaforbidden}
Let $G = SU(n), Sp(n), SO(2n), SO(2n+1)$ and let $(F=G/K, g_F)$ be a nonsymmetric flag manifold represented by a painted diagram with only one black node with corresponding admissible minor $\Delta_r$.
Then the only forbidden monomials of types $(2,3)$ contained in $\Delta_r$ are of one of the following kinds

\begin{equation}\label{forb1}
-\frac{1}{2} Z_{\gamma i} Z_{\alpha j} \bar Z_{\alpha i} \bar Z_{\beta j} \bar Z_{\gamma \beta}
\end{equation}
\begin{equation}\label{forb2}
+\frac{1}{2} Z_{\alpha i} Z_{\gamma j} \bar Z_{\alpha i} \bar Z_{\beta j} \bar Z_{\gamma \beta}
\end{equation}
where $1 \leq i, j \leq r$ and $\alpha, \beta, \gamma > r$ and $Z_{\delta\epsilon}:=Z_{\delta\epsilon}(z)$.
\end{lem}

\begin{proof}
It can be verified by a case-by-case analysis that, under the assumptions, we always have $Z^3 = 0$ and then $\exp(Z) = I + Z + \frac{1}{2}Z^2$. Then, by denoting  $A = {}^T \overline{\exp(Z)} \exp(Z)$, we have

$$A = I + (Z + {}^T \bar Z) + (\frac{1}{2} Z^2 + {}^T \bar Z Z + \frac{1}{2} {}^T \bar Z^2) + (\frac{1}{2} {}^T \bar Z Z^2 + \frac{1}{2} {}^T \bar Z^2 Z) + \frac{1}{4} {}^T \bar Z^2 Z^2$$
Now, by this equality, the very definition of determinant 
$$\Delta_r(A) = \sum_{\sigma \in S_r} s(\sigma) A_{1 \sigma(1)} A_{2 \sigma(2)} \cdots A_{r \sigma(r)}$$ and the fact that $Z_{ij} = 0$ for $i, j \leq r$ we get that a forbidden monomial of type $(2,3)$ can occur in $\Delta_r(A)$ only as one of the following addenda:

\begin{equation}\label{forbdim1}
(\frac{1}{2} {}^T \bar Z^2)_{ij} (\frac{1}{2} {}^T \bar Z Z^2)_{ji} = -\frac{1}{4} Z_{\gamma i} Z_{\beta \gamma} \bar Z_{\alpha i} \bar Z_{j \alpha} \bar Z_{\beta j}
\end{equation}
\begin{equation}\label{forbdim2}
(\frac{1}{2} {}^T \bar Z^2)_{ii} (\frac{1}{2} {}^T \bar Z Z^2)_{jj} = +\frac{1}{4} Z_{\gamma j} Z_{\beta \gamma} \bar Z_{\alpha i} \bar Z_{i \alpha} \bar Z_{\beta j}
\end{equation}
\begin{equation}\label{forbdim3}
( {}^T \bar Z Z)_{ij} (\frac{1}{2} {}^T \bar Z^2 Z)_{ji} = -\frac{1}{2} Z_{\gamma i} Z_{\alpha j} \bar Z_{\alpha i} \bar Z_{\beta j} \bar Z_{\gamma \beta}
\end{equation}
\begin{equation}\label{forbdim4}
( {}^T \bar Z Z)_{ii} (\frac{1}{2} {}^T \bar Z^2 Z)_{jj} = +\frac{1}{2} Z_{\alpha i} Z_{\gamma j} \bar Z_{\alpha i} \bar Z_{\beta j} \bar Z_{\gamma \beta}
\end{equation}

Monomials of kind (\ref{forbdim1}) and (\ref{forbdim2}) are in fact zero because $\bar Z_{j \alpha}$ and $\bar Z_{i \alpha}$ vanish for $i, j \leq r$ under the assumptions.
Then only monomials of kind (\ref{forbdim3}) and (\ref{forbdim4}) are left, which proves the lemma.
\end{proof}

Now, by using the Lemma \ref{lemmaforbidden} it is not hard to show that for nonsymmetric flag manifolds represented by a painted diagram with only one black node the corresponding minor $\Delta_r$ always contains a forbidden monomial of type  $(2,3)$.

More precisely, one has to consider the following cases.

\begin{enumerate}
\item[(1)] $G= Sp(n)$, diagram
\begin{align*}
&&& \underset{\substack{\varepsilon_1 - \varepsilon_2}}{\circ} - \cdots - \underset{\substack{\varepsilon_r - \varepsilon_{r+1}}}{\bullet} - \cdots -  \underset{\substack{\varepsilon_{n-1} - \varepsilon_n}}{\circ} \Leftarrow \underset{\substack{2 \varepsilon_n}}{\circ} &&
\end{align*}
with $1 < r \leq n-1$.
(notice that for $r=1$ we get the symmetric space $\frac{Sp(n)}{Sp(1) \times Sp(n-1)} = \Hyp P^{n-1} \simeq \C P^{2n - 1}$).

\item[(2)] $G= SO(2n+1)$, diagram
\begin{align*}
&&& \underset{\substack{\varepsilon_1 - \varepsilon_2}}{\circ} - \cdots - \underset{\substack{\varepsilon_r - \varepsilon_{r+1}}}{\bullet} - \cdots -  \underset{\substack{\varepsilon_{n-1} - \varepsilon_n}}{\circ} \Rightarrow \underset{\substack{2 \varepsilon_n}}{\circ} &&
\end{align*}
with $1 < r \leq n-1$.

\item[(3)] $G= SO(2n)$, diagram
\begin{align*}
&&& \underset{\substack{\varepsilon_1 - \varepsilon_2}}{\circ} - \cdots - \underset{\substack{\varepsilon_r - \varepsilon_{r+1}}}{\bullet} - \cdots - \underset{\substack{\varepsilon_{n-2} - \varepsilon_{n-1}}}{\overset{\overset{\textstyle\circ_{\varepsilon_{n-1} + \varepsilon_n}}{\textstyle\vert}}{\circ}} \,-\, \underset{\substack{\varepsilon_{n-1} - \varepsilon_n}}{\circ} &&
\end{align*}
with $1 < r \leq n-1$.
\end{enumerate}

Then, one sees that in all these cases the forbidden monomial of type (2,3)
\begin{equation}\label{monomial3cases}
\frac{1}{2} Z_{n 1} Z_{n+1, 2} \bar Z_{n 1} \bar Z_{2n, 2} \bar Z_{n+1, 2n}
\end{equation}
is contained in $\Delta_r$.

Indeed, this monomial is of kind (\ref{forb2}) with $\alpha = n, i =1, \gamma = n+1, j=2, \beta = 2n$ so is contained in $\Delta_r$ by Lemma \ref{lemmaforbidden}. Then,  by a case-by-case analysis and checking all the possibilities for $\alpha, i, \gamma, j, \beta$ it is not hard to see that this monomial does not simplify with other monomials either of kind  (\ref{forb1})  or of kind  (\ref{forb2}) contained in $\Delta_r$. The proof of Theorem \ref{mainteor3} is complete.

\begin{remark}\rm
Note that for $c_1=c_2$ the metric $J_F^*g_F$ on $\C^n$ whose diastasis is given by \eqref{delta12} is KE and hence
one gets examples of exact domains $(\C^n,J_F^*g_F)$ even KE which have a Fubini-Study compactification and  not admitting a \K\ dual. In order to give an explicit example consider  the flag manifold $F=\frac{SU(3)}{S(U(1)^3)}$ of complex dimension $3$. In this case the general homogeneous \K\ metric $g_F$ is given on $\C^3$ by
$$D_0^{J_F^*g_F} = c_1 \log \Delta_1 + c_2  \log \Delta_2,\  c_1, c_2\in\R^+, $$
where  

$$\Delta_1= \log \left[ 1 + |z_1|^2 + |z_2|^2 + \frac{|z_1|^2 |z_3|^2}{4} + \frac{z_2 \bar z_1 \bar z_3}{2} + \frac{\bar z_2 z_1 z_3}{2} \right]$$
and 
$$\Delta_2=\log \left[ 1 + |z_2|^2 + |z_3|^2 + \frac{|z_1|^2 |z_3|^2}{4} - \frac{z_2 \bar z_1 \bar z_3}{2} - \frac{\bar z_2 z_1 z_3}{2} \right]$$
and one can directly check that for $c_1=c_2$ the exact domain $(\C^3, J_F^*g_F)$ has a Fubini-Study compactification given by  $(F, g_F)$, $g_F$ is KE but $(\C^3, J_F^*g_F)$ does not admits a \K\ dual.
\end{remark}

\section{A conjecture}\label{conjecture}

In light of the Theorem \ref{mainteor2}, one might ask   which additional  properties on the metric $g$ 
in Question A would yield a positive answer.
We believe this extra property is the requirement that the metric $g$ is KE.
This is expressed by the following
 
 \vskip 0.3cm
\noindent
{\bf Conjecture A.} {\em Let $(U, g)$ be  an exact domain 
admitting  a \K\ dual $(U^*, g^*)$.
Assume that  there exists $\alpha>0$ such that: 
\begin{itemize}
\item [(a)] 
$(U, \alpha g)$ has a Fubini-Study completion;
\item [(b)]
$(U^*, \alpha g^*)$ has a Fubini-Study compactification.
\end{itemize}
If the metric $g$ is KE, then 
$(U, g)$ is  biholomorphically isometric to  a bounded symmetric  domain $(\Omega, g_\Omega)$.
}

 \vskip 0.3cm
 
We will need the following result, interesting in its own right.
 
 \begin{lem}\label{dualKE}
 Let $(U, g)$ be  an exact domain 
admitting  a \K\ dual $(U^*, g^*)$.
If $g$ is KE with Einstein constant $\lambda$
then $g^*$ is KE with Einstein constant $\lambda^*=-\lambda$.
 \end{lem}
 \begin{proof}
 A direct computation of the Ricci forms $\rho_g$ and $\rho_{g^*}$ for  $g$ and $g^*$ show that they are  related by  
\begin{equation}\label{ricciforms}
\rho_g(z, \bar z)=-\rho_{g^*}(z, -\bar z)
\end{equation}
and the proof easily follows.
 \end{proof}
 
 \begin{remark}\rm\label{dualextr}
From \eqref{ricciforms} one  sees  that the scalar curvatures $\scal_g$ and  $\scal_{g^*}$ 
of the two metrics satisfy  
$$\scal_g(z, \bar z)=-\scal_{g^*}(z, -\bar z).$$
Thus, one deduces that the metric $g$ is extremal (in Calabi's sense\footnote{A metric $g$
is extremal if  the  $(1, 0)$-part of the Hamiltonian vector field associated to the scalar curvature of $g$ is holomorphic.}
\cite{CALABI1982ext}), iff 
$g^*$ is extremal.  
By this fact and since  a finitely projectively induced extremal metric is conjecturally  KE (see  \cite[Conjecture 1]{LOISALISZUDDAS2021extremal} and 
\cite[Th,1.1]{LOISALISZUDDAS2021extremal} for the proof of the validity of this conjecture  for  radial metrics), we believe that the assumption that $g$ is KE in Conjecture A can be be weakened by requiring only that $g$ be extremal.
\end{remark}
 
When $(U, g)=(M_{\Omega, \mu}, g_{\Omega, \mu})$, the following result shows
that  Conjecture A it is true by weakening  condition (b), namely by only requiring the dual \K\  metric $\alpha g_{\Omega, \mu}^*$ 
on $\C^{n+1}$ to be  projectively induced and the KE assumption by requiring $g$ is extremal.

\begin{cor}\label{maincor}
The \K\ metric  $\alpha g^*_{\Omega, \mu}$ is 
projectively induced and extremal  iff  
$M_{\Omega, \mu}=\C H^{n+1}$, $g_{\Omega, \mu}=g_{hyp}$,  and $\alpha\in\Z^+$.
\end{cor}
\begin{proof}
If $g^*_{\Omega, \mu}$ is extremal then by Remark \ref{dualextr} also $g_{\Omega, \mu}$
is extremal.  By \cite[Th.1]{ZEDDA2012CanonicalMetricCH} $g_{\Omega, \mu}$ is forced to be KE
and hence, by Theorem C  in the introduction,  $\mu=\frac{\gamma}{n+1}$,  where $\gamma$
is the genus of $\Omega$.
On the other hand,  it is easily seen that  
\begin{equation}\label{ingamma}
\frac{\gamma}{n+1}\leq 1
\end{equation}
 and equality  holds 
iff $\Omega =\C H^n$ and $\mu=1$
(inequality \eqref{ingamma} can be deduced by looking, for example,  at the  the table in  \cite[pag. 17]{ARAZYbookAsurveyOfInvariant})). 
Hence,  the proof follows  by Proposition \ref{propproj}.
\end{proof}

\vskip 0.3cm

\noindent
{\bf On the necessity of  the KE assumption  in Conjecture A.}
The  following proposition combined  with 
Theorem \ref{mainteor2} shows 
the necessity of  the KE assumption  in Conjecture A.
\begin{prop}
The metric $\hat g_{\Omega, \mu}$ on the CH domain $M_{\Omega, \mu}$ is not  
KE for any value of $\mu>0$.
\end{prop}
\begin{proof}
We use 
\cite[Lemma 5]{WANGYINZHANGROSS2006KEonHartogs} which asserts that 
a \K\  metric $g$ with associated \K\  form   $\omega=\frac{i}{2\pi}\partial\bar\partial \Phi$ on the CH domain
$M_{\Omega, \mu}$ is KE, with Einstein constant $-(n+2)$,  if and only if 

\begin{equation}\label{potentialtype}
\Phi(z, w)=h(X) -\frac{\gamma + \mu}{n+ 2} \log N_\Omega(z, \bar z),
\end{equation}
where $X = \frac{|w|^2}{N_\Omega(z, \bar z)^{\mu}}$ and $h$ satisfies the differential equation

\begin{equation}\label{ODE}
\left( \mu X h'(X) + \frac{\gamma + \mu}{d + 2} \right)^d [X h'(X)]' = k e^{(n+2)h(X)},
\end{equation}
for some $k \in \R$.

Now, on the one hand, 
it is easily seen that  
the \K\ potential
$$\Phi (z, w):=-\frac{\gamma + \mu}{(\mu+1)(d+2)}
\left[\log \left( N_\Omega(z, \bar z)^{\mu} - |w|^2\right)+\log N_\Omega(z, \bar z)\right]$$
for the metric $-\frac{\gamma + \mu}{(\mu+1)(d+2)}\hat g_{\Omega, \mu}$
satisfies (\ref{potentialtype})  with 
$$
h(X) := - \frac{\gamma + \mu}{(\mu+1)(d+2)}  \log(1 - X).
$$
On the other hand,  a simple computation shows that the  latter does not satisfy 
\eqref{ODE}, for any value of $k$.
\end{proof}

\vskip 0.3cm

\noindent
{\bf On the necessity of  condition  (b) in Conjecture A.}
We are indebted to Hishi Hideyuki  for suggesting the following example, which shows that assumption (b) is necessary for the validity of  Conjecture A.
\begin{example}\rm\label{HIDEYUKI}
Let 
$$U_{2,1}:=\left\{\left(W,V\right)\in \operatorname {Sym}\left(2,\C\right)\times M_{2,1}\left(\C\right)\mid A\left(W,V\right) >> 0\right\}\subset \C^5,
$$
where
$$
A\left(W,V\right)=I_2-W \ov{W}-\frac{1}{2} V {\bar V}^t-\frac{1}{2}\left(I_2-W\right)\left(I_2-\ov W\right)^{-1} \ov V V^t \left(I_2-W\right)^{-1}\left(I_2-\ov W\right)
$$
It can be easily seen that $U_{2, 1}$ is acted upon transitively by the group $O(2, \R)\times U(1)$, via the action $(W, V)\mapsto (AW^tA, AVB)$.
 Hence 
$U_{2, 1}$ is a homogeneous bounded domain.  One can also directly verify that
$U_{2, 1}$ is not symmetric and compute 
its Bergman Kernel, which  is given by
$$
K_{U_{2,1}}\left(W,V\right)=C\det A\left(W,V\right)^{-4}
$$
where $C$ is a constant (whose exact value is not important for our aims).  
Consider the Bergman metric $g_B$ on $U_{2, 1}$
whose associated \K\ form is given by:
$$\w_B=-\frac{2i}{\pi}\de\deb \log \det A\left(W,V\right).$$
Notice that $g_B$ is homogeneous and  hence KE \cite[Th.4.1]{KOBAYASHI1959GeomBounded}.
It is not hard to verify  
that the Calabi's diastasis function centred at  the origin of  $g_B$ is given by
$$
D^{g_B}_0(W,V)=-\log \det A\left(W,V\right).
$$
Notice that by applying the \lq\lq dual trick'' \eqref{diastasisstar}, one gets 
$$(D^{g_B}_0)^*(W,V)=\log \det A^*\left(W,V\right)$$
where 
$$A^*\left(W,V\right)=I_2+W \ov{W}+\frac{1}{2} V {\ov V}^t+\frac{1}{2}\left(I_2-W\right)\left(I_2+\ov W\right)^{-1} \ov V V^t \left(I_2-
W\right)^{-1}\left(I_2+\ov W\right).$$
Since the later is  a well-defined  real-valued function on 
$\C^{5}$, 
by the remarks before the proof of Theorem \ref{mainteor3}  the exact domain $(U_{2, 1}, g_B)$
admits a \K\ dual $(U^*_{2, 1}, g^*_B)$, $U^*_{2, 1}\subset \C^5$, such that 
$D_0^{g_B^*}={(D^{g_B}_0)^*}_{|U^*_{2, 1}}$, where $U^*_{2, 1}$
is the maximal domain of definition of $D_0^{g_B^*}$.

Consider the real-valued $(1, 1)$-form $\omega$ 
on $\C^5$ given by
\begin{equation}\label{omegafond}
\omega=\frac{2i}{\pi}\partial\bar\partial\log\det A^*(W, V)
\end{equation}
and let $g$ be the corresponding symmetric two tensor
(thus $g_B^*=g_{|U_{2, 1}^*}$).

Since the  Bergman metric $g_B$ is infinitely projectively induced 
and  it is complete (by homogeneity) it  follows 
that $(U_{2, 1}, \alpha g_B)$ admits a Fubini-Study completion (cf. Remark \ref{rmkcompl}) for all 
$\alpha\in\Z^+$ 
(and hence (a) in Conjecture A is satisfied for such $\alpha$).

We want to prove that the dual \K\ metric $\alpha g_B^*$
on $U_{2, 1}^*\subset \C^5$
is never  finitely projectively induced, for $\alpha >0$.
This will show that assumption (b) in Conjecture A
is necessary for its validity.
In order to achieve this  result
consider the holomorphic curve
$$
\gamma:\C\ni z \mapsto \left(\left(\begin{array}{ll}
0 & z \\
z & 0
\end{array}\right),\binom{z}{0}\right)\in\C^5.
$$
Let  $\Gamma\subset \C$ be   the connected component of the origin,  where $g_\gamma:=\gamma^*g$ is a \K\ metric.
Thus,   
\begin{equation}\label{positive}
h(z):=\frac{\partial^2[\log\det A^*\left(\gamma(z)\right)]}{\partial z\partial\bar z}>0, \ \forall z\in\Gamma.
\end{equation}

By a direct computation or with the aid 
of Mathematica one can show that 
$$
h(x+\frac{i}{2}x)=
\frac{P(x)}{Q(x)},\ \forall x\in\R.
$$
where 
$$P(x)=4\left(12288+4096 x^2-62592 x^4-12320 x^6-19800 x^8-138750 x^{10}+78125 x^{12}\right)$$
and $$Q(x)=\left(128+288 x^2-120 x^4-50 x^6+625 x^8\right)^2>0.$$

Thus, if $x_0$ is the smallest positive root of $P(x)$ it follows by \eqref{positive} that the line segment 
$$\left\{x+\frac{i}{2}x \mid  x \in (-x_0,x_0) \right\}$$
is a subset of $\Gamma$.
Let $W$ be an open  and simply-connected open neighborhood of the origin of $\C$ such that 
$$\left\{x+\frac{i}{2}x \mid  x \in (-x_0,x_0) \right\}\subset W\subset \Gamma$$
Assume now, by a contradiction, that there exists 
$\alpha>0$ such that $\alpha g^*_B$ is finitely projectively induced. 
Then, since $\alpha g^*_{|U_{2, 1}}=\alpha g_B^*$ there exists an 
 open  neighborhood   of the origin contained in $W$ such that $\alpha {g_\gamma}$
 is finitely projectively induced on this neighborhood. 
 
By combining   Calabi's resolvability  \cite[Th.10]{CALABI1953diast} and Calabi's extension \cite[Th.11]{CALABI1953diast}   (since $W$ is simply-connected) also ${\alpha g_\gamma}_{|W}$ is finitely projectively induced by the same finite complex projective space, say $(\C P^N, g_{FS})$.
Now, again with the aid of Mathematica, the holomorphic sectional  curvature $K(z)$ of  
${g_\gamma}_{|W}$ (which in this case is simply the Gaussian curvature)  satisfies 
$$K(x+\frac{i}{2}x)=-\frac{1}{h}\frac{\partial^2\log h}{\partial z\partial\bar z}_{|z=x+\frac{i}{2}x}=\frac{R(x)}{(P(x))^2},\ \forall x\in (-x_0, x_0),$$
with $R(x)$ a polynomial of degree $36$ with positive coefficients and then satisfying  $R(x)>0$.
Thus 
$$\lim_{x\to x_0^-}K\left(x+\frac{i}{2}x\right)=+\infty.$$
This yields the desired contradiction, since by 
the Gauss-Codazzi equation for Kähler submanifolds (see e.g. \cite[Prop. 9.2]{KobNoBook1996Foundation},  the holomorphic sectional curvature of 
${\alpha g_\gamma}_{|W}$ must be bounded from above  by the constant holomorphic sectional curvature of  the  finite dimensional  complex projective space $(\C P^N, g_{FS})$.  This concludes the example.
\end{example}

\begin{remark}\rm\label{rmkhid}
In the previous example we are showing that $\alpha g_B^*$ is not projectively induced, a condition which is  weaker then (b) in Conjecture A (cf. Proposition \ref{weakbgen} below).
\end{remark}

\begin{remark}\rm
Example \ref{HIDEYUKI} should be compared with the proofs of 
Proposition \ref{propproj}  above. 
 By  a straightforward computation the holomorphic sectional curvature $K$ of the metric  $g$ in the proof of the proposition  is bounded for $\mu\geq 2$ (for example, for $\mu=5/2$, $\lim_{|z|\to \infty}K(\xi)=4/5$). Thus, we cannot apply there the argument used in Example \ref{HIDEYUKI}  to conclude that $g$ is not projectively induced and we need to use the full strength of Calabi's theory. Similar considerations apply to the metric $g$ in Proposition \ref{propalphamu}.
 \end{remark}

\begin{remark}\rm
Notice that given  any homogeneous bounded domain $(\Omega, g)$, with $g$
not necessarily KE,
then $\alpha g$ is infinitely projectively induced for all $\alpha$
sufficiently large (see \cite[Th.1.1]{LOIMOSSA2015rmkhom}) and so  $(\Omega, \alpha g)$
admits a Fubini-Study completion (cf. Remark \ref{rmkcompl}).
Thus, Example \ref{HIDEYUKI} stimulates the following  question, which the authors will consider in a future paper:
 if a homogeneous bounded domain  $(\Omega, g)$  admits a \K\ dual $(U^*, g^*)$
 such that $\alpha g^*$ is finitely projectively induced for some 
 $\alpha >0$, is it true that $(U, g)$ is biholomorphically isometric to a bounded symmetric domain?
 \end{remark}

\vskip 0.3cm

\noindent
{\bf On the necessity of  condition  (a) in Conjecture A.}
In order to understand the role of assumption (a) in Conjecture A
we first recall
the following two conjectures in  the context of   flag manifolds.

 \vskip 0.3cm
 
\noindent
{\bf Conjecture B.}
{\em
Let $(M, g)$ be a compact KE manifold such that $g$ is (finitely) projectively induced. Then  $(M, g)$ is biholomorphically isometric to a 
flag manifold $(F, g_F)$.}

\vskip 0.3cm
 
\noindent
{\bf Conjecture C.} {\em
Let $(F, g_F)$ be a flag manifold of complex dimension $n$. Then, for all $p\in F$ there exists a holomorphic embedding
with dense image $J_F:\C^n\rightarrow F$, 
$J_F(0)=p$, 
such that if $(\C^n, J_F^*g_F)$ admits a \K\ dual 
then $(F, g_F)$
is biholomorphically isometric to 
a Hermitian symmetric space of compact type.}

\vskip 0.3cm
Conjecture B is a long-standing conjecture which has been affirmatively solved in codimension $\leq 2$  (\cite{CHERN1967einstein}, \cite{SMITH1967hyper}), 
for complete intersections \cite{HANO1975},  for toric varieties with $n\leq 4$ \cite[Prop.4.2]{AREZZOLOIZUDDAS2012hom} and for $\T^n$-invariant \K\ metrics with $n\leq 6$ 
 in the recent  preprint \cite{MANNOSALIS2024einstein}.
The  assumption of KE in Conjecture B can be weakened 
 by requiring $g$ is a  \K-Ricci soliton, 
 since a finitely projectively induced \K-Ricci solitons is trivial
 \cite[Th.1.1]{LOIMOSSA2021IsomIntoCSF}).
 Notice also that Conjecture B is not valid for infinitely projectively induced KE metrics
(see \cite[Th.1]{LOIZEDDA2011KEsubProjInf} and also \cite{HAOWANGZHANG2015einstein}). For example,  the metric $\alpha g_{\Omega, \mu}$,  on the CH domain $M_{\Omega, \mu}$ is infinitely  projectively induced 
 and KE for $\mu=\frac{\gamma}{n+1}$ and  sufficiently large $\alpha$, as it follows by Theorem C and Proposition \ref{ioemiki}.
Conjecture C represents an extension of Theorem \ref{mainteor3}
to all flag manifolds, not necessarily of classical types.
Notice that for such  manifolds  the existence of 
the embedding $J_F$ remains an open issue.

We now observe  that  the validity of Conjectures B and C implies   the  validity of Conjecture A, 
even without assumption (a) (hence we suspect  assumption (a) can be dropped from Conjecture A).
Indeed, assume that $(U, g)$ is an exact domain of $\C^n$ admitting a \K\ dual $(U^*, g^*)$.
 Then, the    Fubini-Study compactification of $(U^*, \alpha g^*)$ (assumption (b))
is KE by Lemma \ref{dualKE}). Thus, if Conjecture B is valid this compactification turns out to be a  flag manifold $(F, g_F)$
and if also Conjecture C is true, we have $U^*=J_F(\C^n)$ and $(F, g_F)$ is biholomorphically isometric to a  Hermitian symmetric space of compact type
and hence $(U, g)$ is biholomorphically isometric to  a bounded symmetric domain.

One wonders if the conclusion of 
 Conjecture A  is still valid when weakening the assumption 
(b) from the existence of a Fubini-Study compactification
of $(U^*, \alpha g^*)$,  to the requirement that 
the \K\ metric $\alpha g^*$ is finitely projectively induced (cf. Remarks  \ref{weakbCH} and \ref{rmkhid}) .
It is interesting to notice that 
this is indeed the case if one assumes  the validity of Conjectures B and C 
as expressed by the following

\begin{prop}\label{weakbgen}
Let $(U, g)$ be an exact domain which admits a \K\ dual $(U^*, g^*)$.
Assume that there exists $\alpha>0$ such that  $(U, \alpha g)$ has a Fubini-Study completion, $\alpha g^*$
is finitely projectively induced and $g$ is KE.
If Conjectures B and  C are true then $(U, g)$ is biholomorphically isometric to  a bounded symmetric domain.
\end{prop}
\begin{proof}
By Remark \ref{dualKE} we know that $(U^*,\alpha g^*)$
is KE with Einstein constant $\frac{\lambda^*}{\alpha}$, where $\lambda=-\lambda^*$ is the Einstein constant of $g$.
By  \cite[Th.4.5]{HULIN1996sous} (see also 
\cite[Th.1.2]{li2023extensionsextremalkahlersubmanifolds}
for case of  extremal \K\ metrics) since $\alpha g^*$ is finitely projectively induced, there exists a complete KE manifold  $(P, h)$ such that $U^*\subset P$,  $h_{|U^*}=\alpha g^*$, $h$ is KE (with Einstein constant $\frac{\lambda^*}{\alpha}$),  and 
$h$ is still finitely projectively induced. 

We claim that $\lambda^*>0$. 
Indeed $\lambda^*\neq 0$ by \cite[Cor.1.7]{arezzo2024gromovhausdorfflimitsholomorphicisometries} asserting that a Ricci flat metric cannot be finitely projectively induced.
Moreover, $\lambda^*$ cannot be negative otherwise $\lambda>0$ and the  Fubini-Study completion 
of $(U, \alpha g)$ would be compact by Bonnet-Myers theorem, in contrast to the fact that $\alpha  g$ is infinitely projectively induced.
By applying again Bonnet-Myers theorem we deduce that $P$ is compact.
Thus, by the validity of Conjectures B and C, we deduce that  $(P, h)$ is biholomorphically isometric to a Hermitian symmetric space of compact type
$(\Omega_c, \alpha g_{\Omega_c})$ and hence $(U, g)$ is biholomorphically isometric to  the bounded symmetric domain $(\Omega, g_{\Omega})$. 
\end{proof}

\begin{remark}\rm\label{necessitya}
We note that in the proof of the  proposition   we have used assumption  (a) of Conjecture A only to  deduce that $\lambda^*>0$. 
This probably can be obtained by the fact that $\lambda^*<0$ is conjecturally  forbidden by \cite[Conjecture 1]{LOISALISZUDDAS2021extremal} asserting that the finitely projectively induced KE metric $\alpha g^*$ has non-negative Einstein constant.
Notice also  that in the compact case a finitely projectively 
induced KE metric has positive Einstein constant (see \cite{HULIN2000ke}
for a proof), a fact that we cannot use in our proof since $U^*$ is not compact. 
\end{remark}

\section{The case of generalized CH domains}\label{finrmk}
Our  results regarding CH domains can be 
suitably  extended to generalized CH domains.
By a  {\em generalized CH domain based on a bounded symmetric $\Omega$}  we mean the  bounded domain in $\C ^{n+1}$
given by
\begin{equation*}
M_{\Omega, \bar{\mu}}:=\left\{(z, w) \in \Omega \times\mathbb{C}\,  | \, \left|w\right|^2<\prod_{j=1}^sN_{\Omega_j}^{\mu_j}(z_j, \bar{z_j})\right\}
\end{equation*}
where $\Omega=\Omega_1\times\cdots\times\Omega_s$ and each
$\Omega_j\subset \C^{n_j}$, $n=n_1+\cdots +n_s$, is a Cartan domain of genus $\gamma_j$,
$\bar{\mu}=(\mu_1, \dots, \mu_s)$ is a $s$-vector  with positive  entries $\mu_j$ and 
$z=(z_1,\dots  , z_s)\in\Omega$, with $z_j\in \Omega_j\subset \C^{n_j}$, $j=1, \dots ,s$.
A  generalized CH domain  $M_{\Omega, \bar{\mu}}$
can be   equipped with two natural complete 
\K\ metrics $g_{\W,\bar{\mu}}$ and $\hat g_{\W,\bar{\mu}}$ (generalizing 
$g_{\Omega, \mu}$ and  $\hat g_{\Omega, \mu}$ when $\mu\in\R$ and $\Omega$ is irreducible) whose associated \K\ forms are  
$$\wk_{\W,\bar{\mu}}=-\frac{i}{2\pi}\de\deb\log\left( N_\Omega^{\bar{\mu}}(z, \bar z)-|w|^2\right)$$
$$\hat\wk_{\W,\bar{\mu}}=-\frac{i}{2\pi}\de\deb\log\left(N_\Omega^{\bar{\mu}}(z, \bar z)-|w|^2\right)
-\frac{i}{2\pi}\de\deb\log N_\Omega^{\bar{\mu}}(z, \bar z),$$
where 
$N_\Omega^{\bar{\mu}}(z, \bar z)=\prod_{j=1}^sN_{\Omega_j}^{\mu_j}(z_j, \bar{z_j})$.
It is not hard to see that 
exact domains $(M_{\Omega, \bar\mu}, g_{\Omega, \bar\mu})$
and $(M_{\Omega, \bar\mu}, \hat g_{\Omega, \bar\mu})$
admit \K\ duals
$(\C^{n+1}, g^*_{\Omega, \bar\mu})$ and 
$(\C^{n+1}, \hat g^*_{\Omega, \bar\mu})$,  respectively,  
where the corresponding  associated \K\ forms are given by 
$$\wk^*_{\W,\bar{\mu}}=\frac{i}{2\pi}\de\deb\log\left( N_\Omega^{\bar{\mu}}(z, -\bar z)+|w|^2\right)$$
$$\hat\wk^*_{\W,\bar{\mu}}=\frac{i}{2\pi}\de\deb\log\left(N_\Omega^{\bar{\mu}}(z, -\bar z)+|w|^2\right)
+\frac{i}{2\pi}\de\deb\log N_\Omega^{\bar{\mu}}(z, -\bar z),$$
The following proposition extends Propositions  \ref{propproj}, \ref{ioemiki}, \ref{propmod} and \ref{propalphamu}.
Its proof is omitted  since it  can be easily obtained by reducing to the irreducible case.
\begin{prop}\label{propgenCH}
The \K\ metrics $\alpha g_{\Omega, \bar\mu}$ and $\alpha \hat g_{\Omega, \bar\mu}$ are infinitely projectively induced for sufficiently large  $\alpha$.
The \K\ metrics $\alpha g^*_{\Omega, \bar\mu}$ and $\alpha \hat g^*_{\Omega, \bar\mu}$ are finitely projectively induced iff $\alpha\in\Z^+$
and $\mu_j\in\Z^+$, for $j=1, \dots, s$.
\end{prop}

Using this proposition and by noticing that the proofs of  Theorem \ref{mainteor},Theorem \ref{mainteor2} and Proposition  \ref{notrelaxhom} 
work also when  $\Omega$ is an arbitrary bounded symmetric domain, not necessarily irreducible,  one gets the following results for generalized CH domains.

\begin{theor}\label{mainteor4}
The generalized   CH domain $(M_{\Omega, \alpha\bar\mu}, g_{\Omega, \bar\mu})$
admits a Fubini-Study completion for all  sufficiently large $\alpha$. 
Additionally,  $(\C^{n+1}, \alpha g_{\Omega, \bar\mu}^*)$ admits a Fubini-Study compactification for some $\alpha$
iff $\bar\mu=\mu=1$ and  $\Omega=\C H^n$.
 \end{theor}
 
 \begin{theor}\label{mainteor5}
For $\mu_j\in \Z^+$, $j=1, \dots, s$, and  sufficiently large integer  $\alpha$, 
the generalized CH domain
$(M_{\Omega, \bar\mu},  \alpha\hat g_{\Omega, \bar\mu})$
admits  a Fubini-Study completion 
and 
$(\C^{n+1}, \alpha\hat g_{\Omega, \bar\mu}^*)$
admits a Fubini-Study compactification $(P_{\Omega_c, \bar\mu}, g_{P, \bar\mu})$. 
Moreover, if one writes  $\Omega_c=G/K$ then
$P_{\Omega_c, \mu}$
 is a cohomogeneity $1$ $G$-space, the action of $G$ is ordinary 
 (with $\Omega_c$
as the flag associated to the regular orbit)
and $g_{P, \bar\mu}$
is a $G$-invariant \K\  metric.
\end{theor}

Note that for generalized CH domains,  
the extremality of the \K\  metric $g_{\Omega, \bar\mu}$ does not imply it is KE, 
as  in the  case of  CH domains (see the proof of Corollary \ref{maincor}).
Nevertheless, Corollary \ref{maincor} also   extends to generalized CH domains, as expressed 
by the following
\begin{cor}\label{maincorgCH}
The \K\ metric  $\alpha g^*_{\Omega, \bar\mu}$ is 
projectively induced and extremal  iff  
$\bar\mu=\mu=1$,  $\Omega=\C H^n$,
 and $\alpha\in\Z^+$.
 \end{cor}
 \begin{proof}
 By Remark \ref{dualextr} $g^*_{\Omega, \bar\mu}$ is extremal
 iff $g_{\Omega, \bar\mu}$ is extremal.
 On the other hand, 
 by \cite[Th.1.1]{HAO2016genCH} 
the metric $g_{\Omega, \bar\mu}$ is extremal (iff it has constant scalar curvature)
iff 
\begin{equation}\label{sumgammajs}
\sum_{j=1}^s(n+1-\frac{\gamma_j}{\mu_j})n_j=0.
\end{equation}
Assuming that $\alpha g^*_{\Omega, \bar\mu}$ is 
projectively induced,  we have
\begin{equation}\label{varineq}
\frac{\gamma_j}{\mu_j(n+1)}\leq\frac{\gamma_j}{n+1}\leq \frac{\gamma_j}{n_j+1}\leq 1,\ j=1, \dots ,s.
\end{equation}
where the first  inequality follows by 
the fact that $\mu_j\in\Z^+$, for  $j=1, \dots ,s$
(by Proposition \ref{propgenCH})
and the third inequality is given by \eqref{ingamma}, as $\Omega_j$ is irreducible.
By combining \eqref{sumgammajs} and \eqref{varineq}, we deduce that 
\begin{equation}\label{KEgen}
\mu_j=\frac{\gamma_j}{n+1},\ j=1, \dots , s,
\end{equation}
and $n_j=n$, i.e., 
$\Omega$ is  irreducible. Hence,  we can apply  Corollary \ref{maincor} to get the conclusion.
\end{proof}

\begin{remark}\rm
Notice that conditions \eqref{KEgen} are  equivalent to the assumption 
that the metric  $g_{\Omega, \bar\mu}$ (and hence $g^*_{\Omega, \bar\mu}$, by Lemma \ref{dualKE})  is KE (see \cite[Th.1.2]{HAO2016genCH}).
\end{remark}

\bibliographystyle{IEEEtranSA}
\bibliography{Biblio}

\end{document}